\documentclass[amsfonts,12pt]{amsart}

\usepackage{amssymb}
\usepackage{epsfig}
\usepackage{color}
\usepackage{mathtools}
\newcommand{\ee}{\varepsilon}

\newcommand{\R}{{\mathbb R}}
\newcommand{\Z}{{\mathbb Z}}

\newcommand{\cH}{{\mathcal H}}

\newtheorem{thm}{Theorem}[section]

\newtheorem{lem}[thm]{Lemma}
\newtheorem{cor}[thm]{Corollary}

\newtheorem{prop}[thm]{Proposition}

\newtheorem{prob}[thm]{Problem}

\newcommand{\inte}{{\mathrm{int}}\,}

\newcommand{\conv}{{\mathrm{conv}}\,}

\begin{document}
\hfill\today
\bigskip

\title{Reverse and dual Loomis-Whitney-type inequalities}
\author[Stefano Campi, Richard J. Gardner, and Paolo Gronchi]
{Stefano Campi, Richard J. Gardner, and Paolo Gronchi}
\address{Dipartimento di Ingegneria dell'Informazione e di Scienze
Matematiche,
Universit\`a degli Studi di Siena, Via Roma 56, 53100 Siena, Italy}
\email{campi@dii.unisi.it}
\address{Department of Mathematics, Western Washington University,
Bellingham, WA 98225-9063} \email{Richard.Gardner@wwu.edu}
\address{Dipartimento di Matematica "U. Dini", Universit\`a degli Studi di Firenze, Piazza Ghiberti 27, 50122
Firenze, Italy} \email{paolo@fi.iac.cnr.it}
\thanks{Second author supported in
part by U.S.~National Science Foundation Grant DMS-1103612.}
\subjclass[2010]{Primary: 52A20, 52A40; secondary: 52A38} \keywords{convex body, zonoid, intrinsic volume, Loomis-Whitney inequality, Meyer's inequality, Betke-McMullen conjecture, Cauchy-Binet theorem, geometric tomography}

\maketitle

\begin{abstract}
Various results are proved giving lower bounds for the $m$th intrinsic volume $V_m(K)$, $m=1,\dots,n-1$, of a compact convex set $K$ in $\R^n$, in terms of the $m$th intrinsic volumes of its projections on the coordinate hyperplanes (or its intersections with the coordinate hyperplanes).  The bounds are sharp when $m=1$ and $m=n-1$. These are reverse (or dual, respectively) forms of the Loomis-Whitney inequality and versions of it that apply to intrinsic volumes.   For the intrinsic volume $V_1(K)$, which corresponds to mean width, the inequality obtained confirms a conjecture of Betke and McMullen made in 1983.
\end{abstract}

\section{Introduction}
The Loomis-Whitney inequality states that for a Borel set $A$ in $\R^n$,
\begin{equation}\label{LWhit}
{\cH}^n(A)^{n-1}\le \prod_{i=1}^n {\cH}^{n-1}\left(A|e_i^{\perp}\right),
\end{equation}
where $A|e_i^{\perp}$ denotes the orthogonal projection of $A$ on the $i$th coordinate hyperplane $e_i^{\perp}$, and where equality holds when $A$ is a coordinate box.  (See Section~\ref{prelim} for unexplained notation and terminology.)  First proved by Loomis and Whitney \cite{LooW49} in 1949, it is one of the fundamental inequalities in mathematics, included in many texts; see, for example, \cite[Theorem~11.3.1]{BurZ88}, \cite[p.~383]{Gar06}, \cite[Corollary~5.7.2]{Garling07}, \cite[Section~4.4.2]{Had57}, and \cite[Lemma~12.1.4]{Pfe93}.  Since the present article is to some extent a sequel to \cite{CamG11}, we refer to that paper for numerous references to geometrical, discrete, and analytical versions and generalizations of (\ref{LWhit}) and applications to Sobolev inequalities and embedding, stereology, geochemistry, data processing, and compressed sensing.  In addition one may mention Balister and Bollob\'{a}s \cite{BalB12} and Gyarmati, Matolcsi, and Ruzsa \cite{GMR10}, where the Loomis-Whitney inequality finds use in combinatorics and the theory of sum sets, the former paper also citing Han \cite{Han78}, who proved an analogue of the Loomis-Whitney inequality for the entropy of a finite set of random variables; the observation of Bennett, Carbery, and Tao \cite{BCT06} that the multilinear Kakeya conjecture may be viewed as a generalization of the Loomis-Whitney inequality; and applications to group theory by Gromov \cite{Gro08}, graph theory by Madras, Sumners, and Whittington \cite{MSW09}, and data complexity by Ngo, Porat, R\'{e}, and Rudra \cite{NPRR}.

The focus here is on reverse forms of the Loomis-Whitney inequality for compact convex sets---where lower bounds instead of upper bounds are obtained in terms of projections on coordinate subspaces---or on dual forms, where lower bounds are given in terms of intersections with coordinate subspaces.  An example of the latter type is an inequality due to Meyer \cite{Mey88}, which states that if $K$ is a convex body in $\R^n$, then
\begin{equation}\label{Meyer inequality}
V(K)^{n-1}\geq \frac{(n-1)!}{n^{n-1}}\prod_{i=1}^nV_{n-1}(K\cap e_i^\perp),
\end{equation}
with equality if and only if $K$ is a coordinate cross-polytope.  Here $V$ denotes volume, so that whereas (\ref{LWhit}) provides, in particular, an upper bound for the volume of a convex body in terms of volumes of its projections on coordinate hyperplanes, Meyer's inequality gives a lower bound in terms of the volumes of its sections by coordinate hyperplanes.

In \cite{CamG11}, variants of the Loomis-Whitney inequality were found in which an upper bound for $V(K)$ is replaced by an upper bound for the intrinsic volume $V_m(K)$, for some $m\in \{1,\dots,n-1\}$. Our interest is in doing the same for lower bounds.  Intrinsic volumes include as special cases surface area and mean width, corresponding to the cases $m=n-1$ and $m=1$, respectively, up to constant factors depending only on the dimension $n$.

Upper and lower bounds of this sort were first obtained in the pioneering study of Betke and McMullen \cite{BetM83}.  Their motivation was somewhat different, but, as was noted in \cite{CamG11}, it is a consequence of their results that
\begin{equation}\label{BM33}
V_{n-1}(K)\le \sum_{i=1}^nV_{n-1}\left(K|e_i^{\perp}\right),
\end{equation}
with equality if and only if $K$ is a (possibly lower-dimensional) coordinate box.  (Here and for the remainder of the introduction, $K$ is always a compact convex set in $\R^n$.)  Similarly, we observe that it follows from \cite[Theorem~2]{BetM83} that
\begin{equation}\label{SQUARE33}
V_{n-1}(K)^2 \geq \sum_{i=1}^nV_{n-1}(K|e_i^\perp)^2,
\end{equation}
with equality if and only if either $\dim K\le n-1$ or $\dim K=n$ and $K$ is a coordinate cross-polytope.  Since each section is contained in the corresponding projection, the same inequality holds with $V_{n-1}(K|e_i^\perp)$ replaced by $V_{n-1}(K\cap e_i^\perp)$, though the equality condition is then slightly different.  (See Theorem~\ref{BetkeMcMullenSquare};
note that (\ref{Meyer inequality}) is false when sections are replaced by projections, since the left-hand side can then be zero when the right-hand side is positive.)  For this reason we concentrate on lower bounds involving projections for the rest of the introduction.

Campi and Gronchi \cite{CamG11} conjectured a generalization of (\ref{BM33}), namely,
\begin{equation}\label{CG33}
V_{m}(K)\le \frac{1}{n-m}\sum_{i=1}^nV_{m}\left(K|e_i^{\perp}\right),
\end{equation}
where $m=1,\dots,n-1$, with equality if and only if $K$ is a (possibly lower-dimensional) coordinate box.  They proved (\ref{CG33}) when $m=1$ and when $K$ is a zonoid, but for $m\in \{2,\dots,n-2\}$ and general $K$, it remains an open problem.  (Though it is not mentioned in \cite{CamG11}, inequality (\ref{CG33}) for $m=1$ confirms a conjecture of Betke and McMullen, namely, the case $r=1$ and $s=d-1$ of \cite[Conjecture~3(a), p.~537]{BetM83}.) This naturally suggests a corresponding generalization of (\ref{SQUARE33}):
\begin{equation}\label{zonoid33}
V_{m}(K)^2 \geq \frac{1}{n-m}\sum_{i=1}^nV_{m}(K|e_i^\perp)^2.
\end{equation}
In Theorem~\ref{Zonoidm}, we show that (\ref{zonoid33}) holds when $K$ is a zonoid. By a generalized Pythagorean theorem proved in Proposition~\ref{CBm},
equality holds in (\ref{zonoid33}) when $K$ is contained in an $m$-dimensional plane. Exact equality conditions for (\ref{zonoid33}) are complicated to interpret, but we provide a clear geometric description when $m=1$.

It turns out that if $m<n-1$, (\ref{zonoid33}) is false for general $K$, as we prove at the end of Section~\ref{zonoids} for $n=3$ and $m=1$.

A lower bound for the mean width of $K$ in terms of the mean widths of its projections on coordinate hyperplanes was also conjectured by Betke and McMullen in 1983 (the case $r=1$ and $s=d-1$ of \cite[Conjecture~3(b), p.~537]{BetM83}). We show in Theorem~\ref{BMproof} that their conjecture is true by proving the existence of a constant $c_0=c_0(n)$, $n\geq
2$, such that
$$
V_1(K)\ge c_0\sum_{i=1}^n V_1(K|e_i^\perp),
$$
with equality if and only if $K$ is either a singleton or a regular coordinate cross-polytope.

Sharp reverse inequalities of the isoperimetric type are relatively rare
and hard to prove.  Examples can be found in \cite[Remark~9.2.10(ii)]{Gar06} and in \cite{LYZ1}, \cite{LYZ2}, \cite{SchW12}, and the references given in these papers.

Occasionally we take a more general viewpoint, considering estimates of
the $j$th intrinsic volume $V_j(K)$ of $K$ in terms of the $m$th intrinsic volumes of its projections on or intersections with coordinate hyperplanes.  However, for the most part, the difficulty of finding sharp bounds forces us back to the case when $j=m$.

The paper is organized as follows.  After the preliminary Section~\ref{prelim}, lower bounds for $V_{n-1}(K)$ are presented in Section~\ref{n-1}.  The main argument and the case $n=3$ of the Betke-McMullen conjecture is proved in Section~\ref{BMconje}; the long and somewhat technical case $n\ge 4$ is deferred to an appendix in order to maintain the flow of the paper.  Results for zonoids are the topic of Section~\ref{zonoids}.  In Section~\ref{less}, we gather several supplementary results, some involving upper bounds as well as lower bounds.  The final Section~\ref{problems} lists some problems for future research.

\section{Preliminaries}\label{prelim}

\subsection{General notation and basic facts}\label{subsec:notations}
As usual, $S^{n-1}$ denotes the unit sphere and $o$ the origin in Euclidean
$n$-space $\R^n$.  We assume throughout that $n\ge 2$.  The Euclidean norm of $x\in \R^n$ is denoted by $|x|$. If $x,y\in\R^n$, then $x\cdot y$ is the inner product of $x$ and $y$ and $[x,y]$ is the line segment with endpoints $x$ and $y$.  The unit ball in $\R^n$ is $B^n=\{x\in \R^n:|x|\le 1\}$. We write $e_1,\dots,e_n$ for the standard orthonormal basis for $\R^n$. We will write $\inte A$ and $\conv A$ for the interior and convex hull, respectively, of a set $A\subset \R^n$.  The {\em dimension} $\dim A$ of $A$ is the dimension of the affine hull of $A$. The indicator function of $A$ will be denoted by $1_A$.  The (orthogonal) projection of $A$ on a plane $H$ is denoted by $A|H$. If $u\in S^{n-1}$, then $u^\perp$ is the $(n-1)$-dimensional subspace orthogonal to $u$.

A set is {\it $o$-symmetric} if it is centrally symmetric, with center at the
origin, and {\em $1$-unconditional} if it is symmetric with respect to the coordinate hyperplanes.

We write ${\mathcal{H}}^k$ for $k$-dimensional Hausdorff measure in $\R^n$,
where $k\in\{1,\dots, n\}$. The notation $dz$ will always
mean $d{\mathcal{H}}^k(z)$ for the appropriate $k=1,\dots, n$.

We now collect some basic material concerning compact convex sets.  Standard references are the books \cite{Gar06}, \cite{Gru07}, and \cite{Sch93}.

Let $K$ be a compact convex set in $\R^n$.  Then $V(K)$ denotes its {\em volume}, that is, $\cH^k(K)$, where $\dim K=k$.  We write $\kappa_n=V(B^n)=\pi^{n/2}/\Gamma(n/2+1)$ for the volume of the unit ball $B^n$.

A {\em convex body} is a compact convex set with a nonempty interior.

A {\em coordinate box} is a (possibly degenerate) rectangular parallelepiped whose facets are parallel to the coordinate hyperplanes.  A {\em cross-polytope} in $\R^n$ is the convex hull of $k$ mutually orthogonal line segments with a point in common, for some $k\in \{1,\dots,n\}$; it is a {\em coordinate cross-polytope} if these line segments are parallel to the coordinate axes.  The adjective {\em regular} for a polytope is used in the traditional sense, so that in particular, a regular cross-polytope in $\R^n$ has dimension $n$.  We shall write $C^n=\conv\{\pm e_1,\dots,\pm e_n\}$ for the standard regular $o$-symmetric coordinate cross polytope in $\R^n$ and $Q^n=\prod_{i=1}^n[-e_i,e_i]$ for the $o$-symmetric coordinate cube in $\R^n$ with side length 2.

If $m\in \{1,\dots,n-1\}$, the {\em $m$th area measure} of $K$ is denoted by $S_m(K,\cdot)$.  When $m=n-1$, $S_{n-1}(K,\cdot)=S(K,\cdot)$ is the {\em surface area measure} of $K$.  The quantity $S(K)=S(K,S^{n-1})$ is the {\em surface area} of $K$.

If $K$ is a nonempty compact convex set in $\R^n$, then
$$
h_K(x)=\sup\{x\cdot y: y\in K\},
$$
for $x\in\R^n$, defines the {\it support function} $h_K$ of $K$.  Since it is positively homogeneous of degree 1, we shall sometimes regard $h_K$ as a function on $S^{n-1}$.

We collect a few facts and formulas concerning mixed and intrinsic volumes.  A {\em mixed volume} $V(K_{i_1},\dots,K_{i_n})$ is a coefficient in the expansion of $V(t_1K_1+\cdots+t_nK_n)$ as a homogeneous polynomial of degree $n$ in the parameters $t_1,\dots,t_n\ge 0$, where $K_1,\dots,K_n$ are compact convex sets in $\R^n$. The notation $V(K,i;L,n-i)$, for example, means that there are $i$ copies of $K$ and $n-i$ copies of $L$. The quantity
\begin{equation}\label{Viform}
V_i(K)=\frac{1}{c_{n,i}}V\left(K,i;B^n,n-i\right),
\end{equation}
where $i\in\{0,\dots,n\}$ and $c_{n,i}=\kappa_{n-i}/\binom{n}{i}$, is called an {\em intrinsic volume} since it is independent of the ambient space containing the compact convex set $K$.  Then $V_n(K)=V(K)$, $V_{n-1}(K)=S(K)/2$, and
\begin{equation}\label{V1form1}
V_1(K)=\frac{n\kappa_n}{2\kappa_{n-1}}({\text{mean width of~}}K).
\end{equation}
See \cite[Sections~A.3~and~A.6]{Gar06}.  From (\ref{V1form1}), we obtain
\begin{equation}\label{V1form2}
V_1(K)=\frac{1}{\kappa_{n-1}}\int_{S^{n-1}}h_K(u)\,du.
\end{equation}

For the reader's convenience, we state the following form of Minkowski's integral inequality \cite[(6.13.8), p.~148]{HLP}.  If $\mu$ is a measure on a set $X$ and $f_i$, $i=1,\dots,k$, are $\mu$-measurable functions on $X$, then for $p\ge 1$,
\begin{equation}\label{Minkint}
\int_X\left(\sum_{i=1}^kf_i(x)^p\right)^{1/p}\,d\mu(x)\ge
\left(\sum_{i=1}^k\left(\int_Xf_i(x)\,d\mu(x)\right)^p\right)^{1/p},
\end{equation}
with equality if and only if the functions $f_i$ are {\em essentially proportional}.  The latter term means that $f_i(x)=b_ig(x)$ for $\mu$-almost all $x$, all $i=1,\dots,k$, and some $\mu$-measurable function $g$ on $X$ and constants $b_i$, $i=1,\dots,k$.

\subsection{Projections and sections}
Let $K$ be a compact convex set in $\R^n$ and let $m\in\{1,\dots,n-1\}$.  Define ${\mathcal{P}}(K,m)$ to be the class of compact convex sets $L$ in $\R^n$ such that
$$
  V_{m}(L|e_i^{\perp})=V_{m}(K|e_i^{\perp}),
$$
for $i=1,\dots,n$.  Similarly, we let ${\mathcal{S}}(K,m)$ be the class of compact convex sets $L$ in $\R^n$ such that
$$
  V_{m}(L\cap e_i^{\perp})=V_{m}(K\cap e_i^{\perp}),
$$
for $i=1,\dots,n$.

If $j, m\in\{1,\dots,n\}$ and $K$ is any compact convex set in $\R^n$, the $j$th intrinsic volume of bodies in ${\mathcal{S}}(K,m)$ is unbounded, so there are no bodies in ${\mathcal{S}}(K,m)$ of maximal $j$th intrinsic volume.  To see this, for each $i=1,\dots,n$, let $D_i$ be a (possibly degenerate) $(n-1)$-dimensional ball in the part of $e_i^{\perp}$ belonging to the positive orthant, such that $o\not\in D_i$ and $V_{m}(D_i)=V_m(K\cap e_i^{\perp})$.  Then for any $x$ in the interior of the positive orthant, we have $L_x=\conv\{D_1\dots,D_n,x\}\in {\mathcal{S}}(K,m)$ and $V_j(L_x)\rightarrow\infty$ as $|x|\rightarrow\infty$, for each $j=1,\dots,n$.  This fact motivates us to focus on sets of minimal $j$th intrinsic volume in ${\mathcal{S}}(K,m)$.

\begin{lem}\label{newlemmin}
If $K$ is a compact convex set in $\R^n$ and $j, m\in\{1,\dots,n\}$, then there exists a minimizer of $V_j$ in ${\mathcal{S}}(K,m)$. If $j\ge m+2$, then the minimum of $V_j$ in ${\mathcal{S}}(K,m)$ is zero.
\end{lem}

\begin{proof}
The second statement in the lemma and the case $j\ge m+2$ of the first statement follow from the existence of an $(m+1)$-dimensional set in ${\mathcal{S}}(K,m)$.  To see this, choose an $(m+1)$-dimensional plane in $\R^n$ whose intersection $H$ with the positive orthant has dimension $m+1$ and satisfies $V_m(H\cap e_i^\perp)> V_m(K\cap e_i^\perp)$, for $i=1,\dots,n$. For each $i=1,\dots,n$, let $L_i\subset H\cap e_i^\perp$ be such that $V_m(L_i)=V_m(K\cap e_i^\perp)$ and $L_i\cap e_k^\perp = \emptyset$ if $k\neq i$. Then $\conv\{L_1,\dots,L_n\}$ is an $(m+1)$-dimensional set in ${\mathcal{S}}(K,m)$.

Suppose that $j\leq m+1$ and let $a=\min_{1\le i\le n}V_m(K\cap e_i^\perp)$. Let $c\ge 0$ and let
$${\mathcal{M}}={\mathcal{M}}(c)=\{L\in {\mathcal{S}}(K,m): V_j(L)\le c\}.$$
By the inequality \cite[(6.4.7), p.~334]{Sch93} between the intrinsic volumes $V_j$ and $V_k$ for $k\ge j$, there is a constant $d$ such that $V_k(L)\le d$ for each $L\in {\mathcal{M}}$ and $j\le k\le n$.  Let $w>(m+1)d/a$.  We claim that if $L\in {\mathcal{M}}$, then $L\subset [-w,w]^n$.  Indeed, if this is not true, then there is an $x=(x_1,\dots,x_n)\in L$ and $i_0\in \{1,\dots,n\}$ such that $|x_{i_0}|>w$.  If $J=\conv\{x,L\cap e_{i_0}^\perp\}$, then $J\subset L$. Since $J$ is a cone, the formula \cite[(4.5.35), p.~255]{Sch93} from translative integral geometry, with $E_k$ replaced by $e_{i_0}^\perp$, yields
$$V_{m+1}(J)=V_m(L\cap e_{i_0}^\perp)|x_{i_0}|/(m+1).$$
From these facts, we obtain
$$V_{m+1}(L)\ge V_{m+1}(J)> V_m(L\cap e_{i_0}^\perp)w/(m+1)\ge aw/(m+1) > d,$$
contradicting the definition of $d$.  This proves the claim.  As a consequence, the class ${\mathcal{M}}$ is compact in the Hausdorff metric and the existence of the minimizer follows.
\end{proof}

Since intrinsic volumes are monotonic (see, for example, \cite[(A.18), p.~399]{Gar06}) and $K\cap e_i^{\perp}\subset K$ for $i=1,\dots,n$, we have the trivial lower bound
\begin{equation}\label{trivmax}
V_m(K)\ge \max_{1\le i\le n}V_m(K\cap e_i^{\perp}),
\end{equation}
for $m=1,\dots,n-1$.  It is also true, but not trivial, that
\begin{equation}\label{trivmax2}
V_m(K)\ge \max_{1\le i\le n}V_m(K|e_i^{\perp}),
\end{equation}
for $m=1,\dots,n-1$.  This follows from the observation in \cite[p.~556]{CamG11} (where it is stated for $u=e_i$), that
\begin{equation}\label{cgtriv}
V_m(K)\ge V_m(K|u^{\perp}),
\end{equation}
for all $u\in S^{n-1}$.  Of course, (\ref{trivmax2}) is stronger than (\ref{trivmax}) because the obvious containment $K\cap e_i^{\perp}\subset K|e_i^{\perp}$ implies that
$$
V_m\left(K|e_i^{\perp}\right)\ge V_m\left(K\cap e_i^{\perp}\right),
$$
for $i, m=1,\dots,n-1$.  The easy bounds (\ref{trivmax}) and (\ref{trivmax2}) imply that for all $p>0$,
\begin{equation}\label{easynbound}
V_m(K)\ge \left(\frac1n\sum_{i=1}^nV_m(K|e_i^{\perp})^p\right)^{1/p}
\ge \left(\frac1n\sum_{i=1}^nV_m(K\cap e_i^{\perp})^p\right)^{1/p}.
\end{equation}

The so-called Pythagorean inequalities state that for a compact convex set $K$ in $\R^n$ and $m\in \{1,\dots,n-1\}$,
\begin{equation}\label{Pythagorean}
V_m(K|u^{\perp})^2\le \sum_{i=1}^nV_m(K|e_i^{\perp})^2,
\end{equation}
for all $u\in S^{n-1}$. These were proved by Firey \cite{Fir60} (see also \cite[(3), p.~153]{BurZ88} and \cite[Theorem~9.3.8 and Note~9.6]{Gar06}).

We are not aware of an explicit statement and proof of the following result in the literature.

\begin{prop}\label{CBm}
If $m\in \{1,\dots,n-1\}$ and $A$ is a Borel set contained in an $m$-dimensional plane in $\R^n$, then
$$
\cH^m(A)^2= \frac{1}{n-m}\sum_{i=1}^n\cH^m(A|e_i^{\perp})^2.
$$
\end{prop}

\begin{proof}
It is a well-known consequence of the Cauchy-Binet theorem that the following generalized Pythagorean theorem holds (see, for example, \cite{ConB74}):
\begin{equation}\label{eqCauBin}
\cH^m(A)^2= \sum\left\{\cH^m(A|S)^2: {\text{$S$ is an $m$-dimensional coordinate subspace}}\right\}.
\end{equation}
Note that if $1\le i_1<i_2<\cdots<i_{n-m}\le n$ and $S$ is the $m$-dimensional subspace that is the orthogonal complement of the subspace spanned by $e_{i_1},e_{i_2},\dots,e_{i_{n-m}}$, then
$$A|S=\left(\cdots\left((A|e_{i_1}^{\perp})|e_{i_2}^{\perp}\right)|
\cdots\right)|e_{i_{n-m}}^{\perp}.$$
Here, the order of the successive projections of $A$ on the $e_{i_k}^{\perp}$'s can be changed arbitrarily.  Using this and (\ref{eqCauBin}) (twice, the second time with $A$ replaced by $A|e_i^{\perp}$), we obtain
\begin{eqnarray*}
\cH^m(A)^2&=&\sum_{1\le i_1<i_2<\cdots<i_{n-m}\le n}\cH^m\left(\left(\cdots\left((A|e_{i_1}^{\perp})|e_{i_2}^{\perp}\right)|
\cdots\right)|e_{i_{n-m}}^{\perp}\right)^2\\
&=&\frac{1}{n-m}\sum_{i=1}^{n}\sum_{1\le i_1<i_2<\cdots<i_{n-m-1}\le n,~ i\neq i_k}\cH^m\left(\left(\cdots\left((A|e_{i}^{\perp})
|e_{i_1}^{\perp}\right)|
\cdots\right)|e_{i_{n-m-1}}^{\perp}\right)^2\\
&= &\frac{1}{n-m}\sum_{i=1}^n\cH^m(A|e_i^{\perp})^2,
\end{eqnarray*}
since the double sum in the second equation counts each summand in the first sum $n-m$ times.
\end{proof}

\section{The case $m=n-1$}\label{n-1}

After the following lemma, we will apply Meyer's inequality (\ref{Meyer inequality}) to deal with the problem of minimum volume in $\mathcal{S}(K,n-1)$.

\begin{lem}\label{Existence}
If $s_1,\dots,s_n$ are positive real numbers, there is a unique $n$-dimensional $o$-symmetric coordinate cross-polytope $C$ in $\R^n$ such that
$$V_{n-1}(C\cap e_i^\perp)=s_i,$$
for $i=1,2,\dots,n$.
\end{lem}

\begin{proof}
Let $C$ be the $o$-symmetric coordinate cross-polytope defined by
$$C=\conv\{[-t_ie_i,t_ie_i]: i=1,\dots,n\},$$
where $t_i>0$, $i=1,\dots,n$. Then we require that
$$V_{n-1}(C\cap e_i^{\perp})=\frac{2^{n-1}}{(n-1)!}\prod_{k\neq i}t_k=s_i,
$$
for $i=1,\dots,n$.
It is easily checked that the unique solution of this system is given by
$$
  t_i=\frac{1}{2s_i}\left((n-1)!\prod_{k=1}^ns_k\right)^{1/(n-1)},
$$
for $i=1,\dots,n$.  Since $t_i>0$, $i=1,\dots,n$, we have $\dim C=n$.
\end{proof}

\begin{cor}\label{Existence2}
If $K$ is a convex body in $\R^n$ containing the origin in its interior, then there is a unique $n$-dimensional $o$-symmetric coordinate cross-polytope $C$ in ${\mathcal{S}}(K,n-1)$ and $C$ is the unique volume minimizer in ${\mathcal{S}}(K,n-1)$.
\end{cor}

\begin{proof}
Let $s_i=V_{n-1}(K\cap e_i^{\perp})$, $i=1,\dots,n$, and note that $s_i>0$ for all $i$ since $o\in \inte K$.  By Lemma~\ref{Existence}, there is a unique $n$-dimensional $o$-symmetric coordinate cross-polytope $C$ in ${\mathcal{S}}(K,n-1)$.  Furthermore, (\ref{Meyer inequality}) implies that $C$ has minimal volume in the class ${\mathcal{S}}(K,n-1)$ and that $C$ is the unique volume minimizer.
\end{proof}

The previous result is clearly false in general if $o\not\in \inte K$. For example, if $K$ is a ball containing the origin and supported by the hyperplane $e_n^{\perp}$, then $V_{n-1}(K\cap e_i^{\perp})>0$ for $i=1,\dots,n-1$ and $V_{n-1}(K\cap e_n^{\perp})=0$, so no $o$-symmetric coordinate cross-polytope exists in ${\mathcal{S}}(K,n-1)$.

The following result was proved by Betke and McMullen \cite[Theorem~2]{BetM83}.  Their motivation and notation were different to ours. To obtain the proposition as we state it, in \cite[Theorem~2]{BetM83} take $d=n$, $\alpha_i=a_i$, and $u_i=e_i$, $i=1,\dots,n$, and note that the zonotope $Z({\mathcal{L}})$ is then the coordinate box $\sum_{i=1}^n a_i[-e_i,e_i]$.

\begin{prop}\label{generalBetkeMcMullen}
Let $K$ be a compact convex set in $\R^n$ and let $a_1,\dots,a_n$ be positive real numbers.  Then
$$\min \{a_i: i=1,\dots,n\}\leq \frac{1}{V_{n-1}(K)}\sum_{i=1}^na_iV_{n-1}(K|e_i^\perp)\leq \left(\sum_{i=1}^na_i^2\right)^{1/2}.$$
Equality holds on the left (or on the right) if and only if the support of $S(K,\cdot)$ is contained in the set of directions of the contact points of the coordinate box $\sum_{i=1}^n a_i[-e_i,e_i]$ with its inscribed (or circumscribed, respectively) ball.
\end{prop}

\begin{cor}\label{BetkeMcMullen}
If $K$ is a compact convex set in $\R^n$, then
\begin{equation}\label{BetkeMcMullen1}
V_{n-1}(K) \geq \frac{1}{\sqrt{n}}\sum_{i=1}^nV_{n-1}(K|e_i^\perp)
\geq \frac{1}{\sqrt{n}}\sum_{i=1}^nV_{n-1}(K\cap e_i^\perp).
\end{equation}
Equality holds in the left-hand inequality if and only if either $\dim K<n-1$, or $\dim K=n-1$ and $K$ is orthogonal to a diagonal of a coordinate cube, or $\dim K=n$ and $K$ is a regular coordinate cross-polytope.  Equality holds in the right-hand inequality involving $V_{n-1}(K)$ if and only if either $\dim K<n-1$ or $\dim K=n$ and $K$ is an $o$-symmetric regular coordinate cross-polytope.
\end{cor}

\begin{proof}
The right-hand inequality in Proposition~\ref{generalBetkeMcMullen}, with $a_i=1$ for $i=1,\dots,n$, immediately gives the left-hand inequality in (\ref{BetkeMcMullen1}).  If $\dim K<n-1$, both sides of the inequality are zero, and if $\dim K\ge n-1$, the equality condition follows easily from that of Proposition~\ref{generalBetkeMcMullen}.

Suppose that equality holds in the right-hand inequality involving $V_{n-1}(K)$. Then the equality condition for the left-hand inequality applies.  If $\dim K=n-1$ and $K$ is orthogonal to a diagonal of a coordinate cube, then $V_{n-1}(K\cap e_i^\perp)=0$ for $i=1,\dots,n$, so this possibility is eliminated.  Suppose that $\dim K=n$ and $K$ is a regular coordinate cross-polytope.  Since equality holds throughout (\ref{BetkeMcMullen1}) and $V_{n-1}(K|e_i^\perp)\ge V_{n-1}(K\cap e_i^\perp)$ for $i=1,\dots,n$, we have $V_{n-1}(K|e_i^\perp)= V_{n-1}(K\cap e_i^\perp)$ for $i=1,\dots,n$, and it follows easily that $K$ is $o$-symmetric.
\end{proof}

The lower bounds in (\ref{BetkeMcMullen1}) are not always better than the corresponding lower bounds in (\ref{trivmax}) and (\ref{trivmax2}) for $m=n-1$. In fact, the following considerably stronger result can also be obtained from Proposition~\ref{generalBetkeMcMullen}.

\begin{thm}\label{BetkeMcMullenSquare}
If $K$ is a compact convex set in $\R^n$, then
\begin{equation}\label{SQUARE}
V_{n-1}(K)^2 \geq \sum_{i=1}^nV_{n-1}(K|e_i^\perp)^2\geq \sum_{i=1}^nV_{n-1}(K\cap e_i^\perp)^2.
\end{equation}
Equality holds in the left-hand inequality if and only if either $\dim K\le n-1$ or $\dim K=n$ and $K$ is a coordinate cross-polytope.  Equality holds in the right-hand inequality involving $V_{n-1}(K)$ if and only if either $\dim K<n-1$, or $\dim K=n-1$ and $K$ is contained in a coordinate hyperplane, or $\dim K=n$ and $K$ is an $o$-symmetric coordinate cross-polytope.
\end{thm}

\begin{proof}
The right-hand inequality in Proposition~\ref{generalBetkeMcMullen}, with $a_i=V(K|e_i^{\perp})$ for $i=1,\dots,n$, immediately gives the left-hand inequality in (\ref{SQUARE}).  When $\dim K<n-1$, both sides are zero, and when $\dim K=n-1$, equality holds by Proposition~\ref{CBm} with $m=n-1$.  Otherwise, if $\dim K=n$, the equality condition follows easily from that of Proposition~\ref{generalBetkeMcMullen}. The right-hand inequality involving $V_{n-1}(K)$ and its equality condition is then straightforward.
\end{proof}

It is easy to check, by partial differentiation with respect to $a_i$, $i=1,\dots,n$, that the choice of the $a_i$'s in the previous proof is optimal.
In particular, we have
$$\sum_{i=1}^nV_{n-1}(K|e_i^\perp)^2\ge \frac{1}{\sqrt{n}}\sum_{i=1}^nV_{n-1}(K| e_i^\perp),$$
an inequality that follows from the fact that the 2-mean of a finite set of nonnegative numbers is greater than or equal to the 1-mean (average).  Moreover, the bound involving projections in (\ref{SQUARE}) is always at least as good as the easy lower bound (\ref{trivmax2}), since the maximum of a finite set of nonnegative numbers is less than or equal to their $p$th sum for any $p>0$ and in particular when $p=2$; see, for example, \cite[(B.6), p.~414]{Gar06}.

The Pythagorean inequality (\ref{Pythagorean}) for $m=n-1$ and (\ref{SQUARE}) split the inequality (\ref{cgtriv}), since together they imply that
$$V_{n-1}(K|u^{\perp})^2\le \sum_{i=1}^nV_{n-1}(K|e_i^\perp)^2 \le V_{n-1}(K)^2,$$
for all $u\in S^{n-1}$.

The following result follows directly from Lemma~\ref{Existence} and Theorem~\ref{BetkeMcMullenSquare} (compare the proof of Corollary~\ref{Existence2}).

\begin{cor}\label{Minimizer}
If $K$ is a convex body in $\R^n$ containing the origin in its interior, then the unique surface area minimizer in ${\mathcal{S}}(K,n-1)$ is the unique $n$-dimensional $o$-symmetric coordinate cross-polytope in ${\mathcal{S}}(K,n-1)$.
\end{cor}

Now we consider the problem of finding lower bounds for the $j$th intrinsic volume of sets in ${\mathcal{S}}(K,n-1)$ when $j<n-1$.  In view of Corollaries~\ref{Existence2} and~\ref{Minimizer}, it would be reasonable to conjecture that for any convex body $K$ in $\R^n$ containing the origin in its interior and any $j\in \{1,\dots,n-1\}$, the minimizer of the $j$th intrinsic volume in ${\mathcal{S}}(K,n-1)$ is the unique $n$-dimensional $o$-symmetric coordinate cross-polytope in ${\mathcal{S}}(K,n-1)$.  However, it turns out that this is not true in general. Indeed, when $n=3$ and $j=1$, a counterexample is given by
$$K_1=\conv\{B^3\cap e_1^{\perp}, B^3\cap e_2^{\perp}, B^3\cap e_3^{\perp}\}.$$ To see this, note that by (\ref{V1form2}), we have
$$V_1(K_1)=\frac{1}{\pi}\int_{S^2}h_{K_1}(u)\,du.$$
The computation of the latter integral is somewhat tedious, so we just provide a sketch.  We consider the part of $S^2$ lying in the positive octant for which $h_K(u)$ equals the support function of the unit disk in the $yz$-plane.  Using spherical polar coordinate angles $(\theta,\varphi)$, $0\le\theta\le 2\pi$, $0\le\varphi\le\pi$, we find that
$$h_{K_1}(\theta,\varphi)=\left(\sin^2\theta\sin^2\varphi+\cos^2\varphi
\right)^{1/2},$$
for $\pi/4\le \theta\le\pi/2$ and $0\le \varphi\le\arctan(\csc\theta)$.  By symmetry, we then have
$$\int_{S^2}h_{K_1}(u)\,du=48\int_{\frac{\pi}{4}}^{\frac{\pi}{2}}
\int_0^{\arctan(\csc\theta)}
\left(\sin^2\theta\sin^2\varphi+\cos^2\varphi\right)^{1/2}\sin\varphi
\,d\varphi\,d\theta.$$
Using standard substitutions, the inner integral evaluates to
$$\frac{1}{2}-\frac{\sin^2\theta}{\sqrt{2}(1+\sin^2\theta)}-
\frac{\sin^2\theta}{2\cos\theta}\log\left(\frac{(\sqrt{2}+\cos\theta)\sin\theta}{
(\cos\theta+1)\sqrt{1+\sin^2\theta}}\right).$$
Numerical integration then yields  $V_1(K_1)=3.8663...$.

Now recall the definition of the standard $o$-symmetric regular coordinate cross-polytope $C^n$ in $\R^n$ and let $K_2=\sqrt{\pi/2}C^3$.
Since $C^3|e_i^{\perp}$ is a square of side length $\sqrt{2}$, we have $$V_2(K_2\cap e_i^{\perp})=\pi=V_2(K_1\cap e_i^{\perp}),$$
for $i=1,2,3$.  Using the formula \cite[(A.36), p.~405]{Gar06} for the $i$th intrinsic volume of a convex polytope, with $n=3$, $i=1$, and $P=C^3$, we see that
$$V_1(C^3)=\sum_{F\in {\mathcal{F}}_1(C^3)}\gamma(F,C^3)V_1(F),$$
where ${\mathcal{F}}_1(C^3)$ is the set of edges of $C^3$ and $\gamma(F,C^3)$ is the normalized exterior angle of $C^3$ at the edge $F$.  For each of the 12 edges of $C^3$ we have $V_1(F)=\sqrt{2}$ and it is easy to calculate that $\gamma(F,C^3)=\arccos(1/3)/2\pi$.  Therefore
\begin{equation}\label{V1cross}
V_1(C^3)=12\sqrt{2}\frac{\arccos(1/3)}{2\pi},
\end{equation}
so
$$V_1(K_2)=\sqrt{\pi/2}V_1(C^3)=\frac{6\arccos(1/3)}{\sqrt{\pi}}=4.1669...
>V_1(K_1).$$

The minimizer of the $j$th intrinsic volume in ${\mathcal{S}}(K,n-1)$ must clearly be equal to the convex hull of its intersections with the coordinate hyperplanes, but it is not obvious why it should be 1-unconditional.  One can attempt to prove this by considering the 1-unconditional convex body $K^s$ that results from performing successive Steiner symmetrizations on $K$ with respect to each of the coordinate hyperplanes (i.e., $n$ symmetrizations, in any order).  Then $V_j(K^s)\le V_j(K)$, since Shephard \cite{She64} proved that $j$th intrinsic volumes do not increase under Steiner symmetrization. (Shephard does not state this fact explicitly, but it follows from \cite[(6), p.~232]{She64}, with $K_1=\cdots=K_j=K$ and $K_{j+1}=\cdots=K_n=B^n$, and \cite[(17), p.~264]{She64}.)  Also,  $V_{n-1}(K\cap e_i^{\perp})$ does not change under Steiner symmetrization with respect to $e_k^{\perp}$ if $k\neq i$, but if $k=i$ the intersection with $e_k^{\perp}$ may increase.

\section{The case $m=1$: The Betke-McMullen conjecture}\label{BMconje}

The following result confirms a conjecture of Betke and McMullen (the case $r=1$ and $s=d-1$ of \cite[Conjecture~3(b), p.~537]{BetM83}).

\begin{thm}\label{BMproof}
Let $n\ge 3$.  For each compact convex set $K$ in $\R^n$ with $\dim K\ge 1$, define
\begin{equation}\label{BMF}
F(K)=\frac{V_1(K)}{\sum_{i=1}^n V_1(K|e_i^\perp)}.
\end{equation}
Then $F(K)$ is minimal if and only if $K$ is a regular coordinate cross-polytope.
\end{thm}

Note that the previous theorem also holds when $n=2$, by Corollary~\ref{BetkeMcMullen}, but then $F(K)$ is also minimal when $K$ is a line segment orthogonal to a diagonal of a coordinate square.  This explains why we prefer to state Theorem~\ref{BMproof} for $n\ge 3$ and assume throughout the proof that this restriction holds.

We shall present the main argument in a series of lemmas.  This will include the complete proof for $n=3$.  The case $n\ge 4$ of Lemma~\ref{lembm3} required for the full result is contained in the appendix.

\begin{lem}\label{lemreduct}
The functional $F$ defined by \eqref{BMF} attains its minimum in the class of convex bodies $K$ in $\R^n$ with the symmetries of the coordinate cube $Q^n$ (or, equivalently, the regular coordinate cross-polytope $C^n$) satisfying $C^n\subset K\subset Q^n$ and such that $K$ is the convex hull of its intersections with coordinate hyperplanes.
\end{lem}

\begin{proof}
The fact that $F$ attains its minimum in the class of compact convex sets in $\R^n$ follows from a standard compactness argument.

Let $G$ be the group of symmetries of $Q^n=\sum_{i=1}^n[-e_i,e_i]$ and let $|G|$ denote its cardinality. For every compact convex set $K$ in $\R^n$ and $g\in G$, let $gK$ be the image of $K$ under $g$.  Then the $G$-symmetral of $K$ is the set
\begin{equation}\label{Gsym}
K^G = \frac{1}{|G|}\sum_{g\in G}gK.
\end{equation}
Using the rigid motion covariance of Minkowski addition, we have, for every $h\in G$,
$$
hK^G=h\left(\frac{1}{|G|}\sum_{g\in G}gK\right)=\frac{1}{|G|}\sum_{g\in G}h(gK)=\frac{1}{|G|}\sum_{g\in G}(hg)K=\frac{1}{|G|}\sum_{g'\in G}g'K = K^G.
$$
It follows that $K^G$ has the symmetries of $Q^n$.

We claim that $F(K^G)=F(K)$.  To see this, note that by (\ref{Viform}) with $i=1$, we have
\begin{equation}\label{ff1}
V_1(K)=\frac{n}{\kappa_{n-1}}V\left(K,1;B^n,n-1\right).
\end{equation}
Also, by \cite[(A.43), p.~407]{Gar06} with $i=n-1$, $K_1=K$, $K_2=\cdots=K_{n-1}=B^n$, and $u_1=e_i$, we have
\begin{equation}\label{ff2}
V_1\left(K|e_i^{\perp}\right)=\frac{1}{n}V\left(K,1;[0,e_i],1;B^n,n-2\right).
\end{equation}
Summing (\ref{ff2}) over $i$, taking into account the multilinearity of mixed volumes (see, for example, \cite[(A.16), p.~399]{Gar06}), and using the resulting equation and (\ref{ff1}), we obtain
\begin{equation}\label{FunK}
    F(K)=c_1\frac{V\left(K,1;B^n,n-1\right)}{V\left(K,1;Q^n,1;B^n,n-2\right)},
\end{equation}
for some constant $c_1=c_1(n)$. For every $g\in G$, we have $gB^n=B^n$, so by the invariance of a mixed volume under a rigid motion of its arguments (see, for example, \cite[(A.17), p.~399]{Gar06}), we have
\begin{equation}\label{fkg1}
V\left(gK,1;B^n,n-1\right)=V\left(K,1;g^{-1}B^n,n-1\right)=V\left(K,1;B^n,n-1
\right).
\end{equation}
Similarly, using $gQ^n=Q^n$, we get
\begin{equation}\label{fkg2}
V\left(gK,1;Q^n,1;B^n,n-2\right)=V\left(K,1;Q^n,1;B^n,n-2\right).
\end{equation}
Substituting (\ref{fkg1}) and (\ref{fkg2}) into (\ref{FunK}), we obtain $F(gK)=F(K)$.
Finally, $F(K^G)=F(K)$ follows from this, the multilinearity of mixed volumes, and the definition (\ref{Gsym}) of $K^G$.  This proves the statement in the lemma regarding symmetries.

Suppose that $F$ attains its minimum at $K$. The function $F$ is invariant under dilatations, so if $K$ has the symmetries of $Q^n$, we can assume that it contains the points $\pm e_i$, $i=1,\dots,n$, in its boundary.  It is then clear from the symmetries of $K$ that $C^n\subset K\subset Q^n$.  Since $K$ must also be 1-unconditional, we have $K|e_i^{\perp}=K\cap e_i^{\perp}$ for $i=1,\dots,n$, and then clearly we may also assume that $K=\conv\{K\cap e_i^{\perp}: i=1,\dots,n\}$.
\end{proof}

\begin{lem}\label{leman}
Let $C^n\subset K\subset Q^n$ be a convex body in $\R^n$ with the symmetries of $Q^n$ and such that $K$ is the convex hull of its intersections with the coordinate hyperplanes.  Then
\begin{equation}\label{F(K)=x}
F(K)=\frac{2\kappa_{n-2}}{\kappa_{n-1}}\frac{\int\limits_0^
    {\frac{1}{\sqrt{n-1}}}
    \int\limits_{x_{n-1}}^{\sqrt{\frac{1-x_{n-1}^2}{n-2}}}\cdots
    \int\limits_{x_3}^{\sqrt{\frac{1-x_3^2-\cdots -x_{n-1}^2}{2}}}h_K(x_1, \dots, x_{n-1},0)p(x_{n-1})/x_1\,dx_2\cdots dx_{n-1}}
    {\int\limits_0^{\frac{1}{\sqrt{n-1}}}
    \int\limits_{x_{n-1}}^{\sqrt{\frac{1-x_{n-1}^2}{n-2}}}\cdots
    \int\limits_{x_3}^{\sqrt{\frac{1-x_3^2-\dots -x_{n-1}^2}{2}}}h_K(x_1, \dots, x_{n-1},0)/x_1 \,dx_2\cdots dx_{n-1}},
\end{equation}
where $x_1=\sqrt{1-x_2^2-\cdots-x_{n-1}^2}$ and $p$ is a nonnegative, increasing, continuous function.
\end{lem}

\begin{proof}
Since the assumptions on $K$ force it to be 1-unconditional, we have $K|e_i^{\perp}=K\cap e_i^{\perp}$, $i=1,\dots,n$. Hence, $K$ is the convex hull of its projections on the coordinate hyperplanes, from which we obtain
\begin{equation}\label{SupportFunctionConvexHull}
h_K(x)=\max_{1\leq i \leq
n} h_{K|e_i^\perp}(x)=\max_{1\leq i \leq n} h_{K}(x|e_i^\perp),
\end{equation}
for all $x\in \R^n$.

In view of the symmetries of $K$, $V_1\left(K|e_i^\perp\right)$ is the same for $i=1,\dots,n$, so identifying $e_i^{\perp}$ with $\R^{n-1}$ and using first (\ref{V1form2}) with $n$ replaced by $n-1$ and then the symmetries of $K$ again, we have
\begin{equation}\label{viperp}
\sum_{i=1}^n V_1(K|e_i^\perp) =
\frac{n}{\kappa_{n-2}}\int_{S^{n-1}\cap e_n^\perp} h_K(u)\, du =
\frac{n!\, 2^{n-1}}{\kappa_{n-2}}\int_{\Omega\cap e_n^\perp}h_K(u)\, du,
\end{equation}
where
\begin{equation}\label{omdef}
\Omega=S^{n-1}\cap\{(x_1,\dots,x_n)\in \R^n: x_1\geq x_2\geq \dots \geq x_{n} \geq 0\}.
\end{equation}
By (\ref{V1form2}) and (\ref{SupportFunctionConvexHull}), we obtain
\begin{equation}\label{gg1}
V_1(K)=\frac{1}{\kappa_{n-1}}\int_{S^{n-1}}\max_{1\leq i\leq n} h_K(u|e_i^\perp)\, du
=\frac{n!\, 2^n}{\kappa_{n-1}}\int_{\Omega}h_K(u|e_n^\perp)\,du.
\end{equation}
If $u\in \Omega$, then $u=\sin\varphi \,v+ \cos\varphi \,e_n$ for some $v\in \Omega\cap e_n^\perp$ and $0\le \varphi\le \pi/2$.  Then by (\ref{omdef}), we have
$$\sin\varphi(v\cdot e_{n-1})=u\cdot e_{n-1}=u_{n-1}\geq u_n=\cos\varphi\ge 0,$$
from which it follows that $\varphi\in [\pi/2-\arctan(v\cdot e_{n-1}),\pi/2]$. Since $h_K$ is positively homogeneous of degree 1, (\ref{gg1}) becomes
\begin{align}\label{v1ff}
V_1(K)&= \frac{n!\, 2^n}{\kappa_{n-1}}\int_
{\Omega\cap e_n^\perp}\int_{\pi/2-\arctan(v\cdot e_{n-1})}^{\pi/2}  h_K(\sin\varphi \, v)\sin^{n-2} \varphi \, d\varphi\, dv\nonumber \\
&= \frac{n!\, 2^n}{\kappa_{n-1}}\int_{\Omega\cap e_n^\perp} h_K(v)\int_{\pi/2-\arctan(v\cdot e_{n-1})}^{\pi/2} \sin^{n-1}\varphi \,d\varphi\,dv \nonumber\\
&= \frac{n!\, 2^n}{\kappa_{n-1}}\int_{\Omega\cap e_n^\perp} h_K(v)p(v\cdot e_{n-1})\, dv,
\end{align}
where
$$
p(t)=\int_{\pi/2-\arctan t}^{\pi/2} \sin^{n-1}\varphi \, d\varphi,
$$
a nonnegative, increasing, continuous function of $t$. Substituting (\ref{viperp}) and (\ref{v1ff}) into (\ref{BMF}), we obtain
\begin{equation}\label{F(K)=u}
F(K)=\frac{2\kappa_{n-2}}{\kappa_{n-1}}\frac{\int_{\Omega\cap e_n^{\perp}} h_K(u) p(u\cdot e_{n-1})\,du}{\int_{\Omega\cap e_n^{\perp}} h_K(u)\,du}.
\end{equation}
We rewrite the integrals in (\ref{F(K)=u}) as integrals over the graph of the function
\begin{equation}\label{x1eq}
x_1=f(x_2,\dots,x_{n-1})=\sqrt{1-x_2^2-\cdots-x_{n-1}^2}.
\end{equation}
The required formula for a surface integral is given explicitly in \cite[Equation (24)]{Ben99}.  Here, the Jacobian is $\sqrt{1+|\nabla f|^2}=1/x_1$.
Let $\Omega_1$ be the projection of $\Omega\cap e_n^{\perp}$ onto $e_1^{\perp}$.  Then $\Omega_1$ is determined by the inequalities
\begin{eqnarray}\label{inarray}
  0\leq &x_{n-1}&\leq \frac{1}{\sqrt{n-1}},\\
  x_{n-1}\leq &x_{n-2}&\leq \sqrt{\frac{1-x_{n-1}^2}{n-2}},\nonumber\\
  x_{n-2}\leq &x_{n-3}&\leq \sqrt{\frac{1-x_{n-2}^2-x_{n-1}^2}{n-3}},\nonumber\\
   & \vdots & \nonumber\\
  x_3\leq &x_2 &\leq \sqrt{\frac{1-x_3^2-x_4^2-\dots -x_{n-1}^2}{2}}.\nonumber
\end{eqnarray}
The equation (\ref{F(K)=x}) results immediately.
\end{proof}

The following result is an inequality of the Chebyshev type; see, for example, \cite[Theorem~236, p.~168]{HLP}.

\begin{lem}\label{Cheblem}
Let $f:[a,b]\to \R$ be continuous and with zero average over $[a,b]$, and suppose that there exists a $c\in [a,b]$ such that $f\le 0$ on $[a,c]$ and $f\ge 0$ on $[c,b]$.  If $g:[a,b]\to \R$ is nonnegative, increasing, and continuous, then $\int_a^bf(t)g(t)\,dt\ge 0$.
\end{lem}

\begin{proof}
Since $g$ is nonnegative and increasing on $[a,b]$, we have $0\le g(t)\le g(c)$ on $[a,c]$ and $g(t)\ge g(c)\ge 0$ on $[c,b]$.  The assumptions on $f$ and $g$ imply that
\begin{eqnarray*}
\int_a^bf(t)g(t)\,dt&=&\int_a^cf(t)g(t)\,dt+\int_c^bf(t)g(t)\,dt\\
&\ge &g(c)\int_a^cf(t)\,dt+g(c)\int_c^bf(t)\,dt=g(c)\int_a^bf(t)\,dt=0.
\end{eqnarray*}
\end{proof}

Recall that $\Omega_1$ is the region defined by (\ref{inarray}) and define $S(t)=\Omega_1\cap\{x\in\R^n: x\cdot e_{n-1}=t\}$.  For $n\ge 4$, define
\begin{equation}\label{jj}
J_K(x_{n-1})=\begin{cases}
\frac{\int_{S(x_{n-1})}h_K(x_1,\dots,x_{n-1},0)/x_1\,dx_2\cdots dx_{n-2}}{\int_{S(x_{n-1})}\,dx_2\cdots dx_{n-2}},& {\text{if $0\le x_{n-1}<1/\sqrt{n-1}$,}}\\[1ex]
h_K(1,\dots,1,0),& {\text{if $x_{n-1}=1/\sqrt{n-1}$}},
\end{cases}
\end{equation}
where $x_1$ is defined by \eqref{x1eq}.  For $n=3$, let
\begin{equation}\label{jj3def}
J_K(x_2)=h_K(x_1,x_2,0)/x_1,
\end{equation}
for $0\le x_2\le 1/\sqrt{2}$.  The function $J_K$ is continuous on $[0,1/\sqrt{n-1}]$.  To see this, note that for $0\leq x_{n-1} < 1/\sqrt{n-1}$, $J_K(x_{n-1})$ is the average of $h_K/x_1$ over $S(x_{n-1})$, while the value of $h_K/x_1$ at the singleton $S(1/\sqrt{n-1})=\{(1/\sqrt{n-1},\dots, 1/\sqrt{n-1},0)\}$ is $h_K(1,\dots,1,0)$.

\begin{lem}\label{lemdiff}
Let $K$ be as in Lemma~\ref{leman}.  If the function $J_K$ be defined by \eqref{jj} and \eqref{jj3def} is increasing, then $F(K)\ge F(C^n)$, with equality if and only if $K=C^n$.
\end{lem}

\begin{proof}
Note firstly that $h_{C^n}(x)=x_1$ for all $x=(x_1,\dots,x_n)$ such that $0=x_n\le x_{n-1}\le \cdots\le x_1\le 1$ and hence $h_{C^n}(x)=x_1$ on $\Omega_1$. For $n\ge 4$, let
\begin{align*}
\lefteqn{G_K(x_{n-1})=\int_{S(x_{n-1})}\left(\frac{h_K(x_1,\dots,x_{n-1},0)/x_1}
{\int_{\Omega_1}
h_K(x_1,\dots,x_{n-1},0)/x_1\,dx_2\cdots dx_{n-1}}-\right.}\\
& &\left.\frac{1}{\int_{\Omega_1}
\,dx_2\cdots dx_{n-1}}\right)\,dx_2\cdots dx_{n-2}.
\end{align*}
For $n=3$, let
$$G_K(x_2)=\frac{h_K(x_1,x_2,0)/x_1}{\int_{\Omega_1}
h_K(x_1,x_2,0)/x_1\,dx_2}-\frac{1}{\int_{\Omega_1}
\,dx_2}.$$
Suppose that $n\ge 3$.  In view of (\ref{F(K)=x}), the inequality $F(K)\ge F(C^n)$ is equivalent to
\begin{equation}\label{intppp}
\int_0^{1/\sqrt{n-1}}
G_K(x_{n-1})p(x_{n-1})\,dx_{n-1}\ge 0.
\end{equation}
From (\ref{jj}) and the definition of $G_K$, we see that $G_K(x_{n-1})\le 0$ or $G_K(x_{n-1})\ge 0$ according as $J_K(x_{n-1})\le I$ or $J_K(x_{n-1})\ge I$, respectively, where
\begin{equation}\label{III}
I=\frac{\int_{\Omega_1}
h_K(x_1,\dots,x_{n-1},0)/x_1\,dx_2\cdots dx_{n-1}}{\int_{\Omega_1}
\,dx_2\cdots dx_{n-1}}.
\end{equation}
Now $G_K$ is continuous and its definition implies that its average over $[0,1/\sqrt{n-1}]$ is zero. Therefore if $J_K$ is increasing, then for some $0\le t_0\le 1/\sqrt{n-1}$, we have $G_K\le 0$ on $[0,t_0]$ and $G_K\ge 0$ on $[t_0,1/\sqrt{n-1}]$.  Since $p$ is nonnegative, increasing, and continuous, (\ref{intppp}) follows from Lemma~\ref{Cheblem} with $f$ and $g$ replaced by $G_K$ and $p$, respectively.

If $F(K)=F(C^n)$, then by (\ref{F(K)=x}),
$$\int_0^{1/\sqrt{n-1}}
G_K(x_{n-1})p(x_{n-1})\,dx_{n-1}= 0.$$
Since $G_K\leq 0$ on $[0,t_0]$, $G_K\geq 0$ on $[t_0,1/\sqrt{n-1}]$, and $G_K$ has
zero average over $[0,1/\sqrt{n-1}]$, the fact that $p$ is
nonnegative and strictly increasing implies that the inequality in (\ref{intppp}) is strict, yielding a contradiction unless $G_K$ vanishes on $[0,1/\sqrt{n-1}]$.  It follows that $J_K(x_{n-1})=I$, for $0\le x_{n-1}\le 1/\sqrt{n-1}$, where $I$ is as in (\ref{III}). In particular,
$J_K(0)=J_K(1/\sqrt{n-1})=h_K(1,\dots,1,0)$. Now $K$ has the same symmetries as $Q^n$, so for each $i\in \{1,\dots,n-1\}$, $h_K$ is convex and even as a function of
$x_i$ and hence increases with $x_i\in [0,1]$. Therefore
$h_K(x_1, ..., x_{n-1}, 0) / x_1 \leq h_K(x_1, x_1,..., x_1, 0) / x_1
= h_K(1, 1,..., 1, 0)$.
Since $J_K(0)$ is the average of $h_K/x_1$ on $S(0)$,  it
follows that $h_K/x_1$ coincides with $h_K(1,\dots, 1, 0)$ on
$S(0)$. In particular, $h_K(1,\dots,1,0)=h_K(1,0,\dots,0)=1$. This
means that the face of $C^n$ orthogonal to $(1,\dots,1,0)$ supports
$K$. The assumptions on $K$ inherited from those in Lemma~\ref{leman} imply that $K=C^n$.
\end{proof}

\begin{lem}\label{lembm3}
Let $K$ be as in Lemma~\ref{leman}.  Then the function $J_K$ defined by \eqref{jj} is increasing.
\end{lem}

\begin{proof}
As we observed in the proof of the previous lemma, the fact that $K$ has the same symmetries as $Q^n$ implies that for each $i\in \{1,\dots,n-1\}$, $h_K$ is convex and even as a function of $x_i$ and hence increases with $x_i\in [0,1]$.

Let $n=3$.  By Lemma~\ref{lemdiff}, it suffices to show that $J_K(x_2)=h_K(x_1,x_2,0)/x_1$
is increasing for $x_2\in [0,1/\sqrt{2}]$, where $h_K(x_1,x_2,0)/x_1=h_K(1,x_2/\sqrt{1-x_2^2},0)$. This is true because $h_K$ is an increasing function of its
second argument in $[0,1]$ and $x_2/\sqrt{1-x_2^2}$ is increasing for $x_2\in [0,1)$.

The case $n\ge 4$ is proved in the appendix.
\end{proof}

\noindent{\em Proof of Theorem~\ref{BMproof}}.
Let $K$ be a compact convex set in $\R^n$ with $\dim K\ge 1$.  By Lemma~\ref{lemreduct}, we may assume that $K$ satisfies the hypotheses of Lemma~\ref{leman}.  Then Lemmas~\ref{lemdiff} and~\ref{lembm3} imply that $F(K)\ge F(C^n)$, with equality if and only if $K=C^n$.
\qed

\begin{cor}\label{corBMconj}
Let $n\ge 3$.  If $K$ is a compact convex set in $\R^n$, then there is a constant $c_0=c_0(n)$ such that
\begin{equation}\label{BetkeMcMullenV1}
V_{1}(K) \geq c_0\sum_{i=1}^nV_{1}(K|e_i^\perp)
\geq c_0\sum_{i=1}^nV_{1}(K\cap e_i^\perp),
\end{equation}
with equality in either inequality involving $V_{1}(K)$ if and only if either $\dim K=0$ or $\dim K=n$ and $K$ is an $o$-symmetric regular coordinate cross-polytope (or one of its translates, in the case of the left-hand inequality).
\end{cor}

\begin{proof}
Theorem~\ref{BMproof} gives the left-hand inequality in (\ref{BetkeMcMullenV1}) and its equality condition.  The right-hand inequality involving $V_{1}(K)$ and its equality condition follow trivially since $V_{1}(K|e_i^\perp)\ge V_{1}(K\cap e_i^\perp)$ for $i=1,\dots,n$.
\end{proof}

With modified equality conditions, the previous corollary also holds when $n=2$, by Corollary~\ref{BetkeMcMullen}.

\section{The case $m<n-1$: Results for zonoids}\label{zonoids}

\begin{thm}\label{Zonoidm}
Let $K$ be a zonoid in $\R^n$ and let $m\in\{1,\dots,n-2\}$.  Then
\begin{equation}\label{eq1}
V_m(K)^2\geq \frac{1}{n-m}\sum_{i=1}^nV_m(K|e_i^{\perp})^2.
\end{equation}
\end{thm}

\begin{proof}
Suppose that $m\in\{1,\dots,n-2\}$ and that $K$ is a zonoid with generating measure $\mu_K$ in $S^{n-1}$ (see \cite[p.~149]{Gar06}).  We use the formula \cite[Theorem~5.3.3]{Sch93} for the $m$th intrinsic volume of a zonoid (twice, once for $K$ and once for the zonoid $K|e_i^{\perp}$), Proposition~\ref{CBm} (with $A=\sum_{k=1}^m
[o,\omega_k]$), and Minkowski's integral inequality (\ref{Minkint}) (with $p=2$ and $k=n$), to obtain
\begin{eqnarray*}
\lefteqn{V_m(K)}\\
&=&\frac{2^m}{m!}\int_{S^{n-1}}\cdots\int_{S^{n-1}}V_m\left(\sum_{k=1}^m
[o,\omega_k]\right)\,d\mu_K(\omega_1)\cdots d\mu_K(\omega_m)\\
&=&\frac{2^m}{m!}\int_{S^{n-1}}\cdots\int_{S^{n-1}}\left(\frac{1}{n-m}
\sum_{i=1}^n V_m\left(\left(\sum_{k=1}^m
[o,\omega_k]\right)|e_i^{\perp}\right)^2\right)^{1/2}\,d\mu_K(\omega_1)\cdots d\mu_K(\omega_m)\\
&\ge &\left(\frac{1}{n-m}\sum_{i=1}^n
\left(\frac{2^m}{m!}\int_{S^{n-1}}\cdots\int_{S^{n-1}}
V_m\left(\left(\sum_{k=1}^m
[o,\omega_k]\right)|e_i^{\perp}\right)d\mu_K(\omega_1)\cdots d\mu_K(\omega_m)\right)
^2\right)^{1/2}\\
&=&\left(\frac{1}{n-m}\sum_{i=1}^n
\left(\frac{2^m}{m!}\int_{S^{n-1}}\cdots\int_{S^{n-1}}
V_m\left(\sum_{k=1}^m
[o,\omega_k|e_i^{\perp}]\right)\,d\mu_K(\omega_1)\cdots d\mu_K(\omega_m)\right)
^2\right)^{1/2}\\
&=&\left(\frac{1}{n-m}\sum_{i=1}^n
\left(\frac{2^m}{m!}\int_{S^{n-1}}\!\!\!\!\!\!\!\cdots\int_{S^{n-1}}
\!\!\!V_m\!\!\left(\sum_{k=1}^m
\left[o,\frac{\omega_k|e_i^{\perp}}{|\omega_k|e_i^{\perp}|}\right]\right)
\prod_{k=1}^m|\omega_k|e_i^{\perp}|\,d\mu_K(\omega_1)\cdots d\mu_K(\omega_m)\right)
^2\right)^{1/2}\\
&=&\left(\frac{1}{n-m}\sum_{i=1}^n\left(\frac{2^m}{m!}\int_{S^{n-1}\cap e_i^{\perp}}\cdots\int_{S^{n-1}\cap e_i^{\perp}}V_m\left(\sum_{k=1}^m
[o,\theta_k]\right)\,d\mu_{K|e_i^{\perp}}(\theta_1)\cdots d\mu_{K|e_i^{\perp}}(\theta_m)\right)^2\right)^{1/2}\\
&=&\left(\frac{1}{n-m}\sum_{i=1}^nV_m(K|e_i^{\perp})^2\right)^{1/2}.
\end{eqnarray*}
Here $\mu_{K|e_i^{\perp}}$ denotes the generating measure in $S^{n-1}\cap e_i^{\perp}$ of the zonoid $K|e_i^{\perp}$. The penultimate equality in the previous display is a consequence of the formula (see, for example, \cite[(10)]{HugS11})
$$\mu_{K|e_i^{\perp}}(A)=\int_{S^{n-1}\setminus\{\pm e_i\}}1_A\left(
\frac{u|e_i^{\perp}}{|u|e_i^{\perp}|}\right)|u|e_i^{\perp}|\,d\mu_K(u),$$
where $1_A$ is the characteristic function of an arbitrary Borel set $A$ in $S^{n-1}$.  This proves the inequality (\ref{eq1}).
\end{proof}

By Proposition~\ref{CBm}, equality holds in (\ref{eq1}) when $\dim K\le m$.  Otherwise, if equality holds, then equality holds in the previous displayed inequality.  A direct consequence of the equality condition for Minkowski's integral inequality (\ref{Minkint}) is that equality in (\ref{eq1}) holds if and only if there are constants $b_i$, $i=1,\dots,n$, and a function $g(\omega_1,\dots,\omega_m)$ which is measurable with respect to the product measure $\nu_K=\mu_K\times\cdots\times\mu_K$ on $\left(S^{n-1}\right)^m$, such that
\begin{equation}\label{finaleqz}
V_m\left(\left(\sum_{k=1}^m
[o,\omega_k]\right)|e_i^{\perp}\right)=b_ig(\omega_1,\dots,\omega_m),
\end{equation}
for all $i=1,\dots,n$ and $\nu_K$-almost all $(\omega_1,\dots,\omega_m)\in \left(S^{n-1}\right)^m$.  When $m=1$, this condition assumes the following more explicit form.

\begin{cor}\label{Zonoidmcor}
Let $K$ be a zonoid in $\R^n$ with $\dim K\ge 1$ and let $m=1$.  Then \eqref{eq1} holds with equality if and only if either $\dim K=1$ or $\dim K>1$ and $K$ is a zonotope of the form \begin{equation}\label{zon1eq}
K=\sum_{r=1}^{2^n}[o,a_r(\ee_1u_1,\dots,\ee_nu_n)],
\end{equation}
where $u=(u_1,\dots,u_n)\in S^{n-1}$, $a_r\ge 0$, and $\ee_i=\pm 1$, $i=1,\dots,n$.
\end{cor}

\begin{proof}
If $m=1$ and $\dim K>1$, the equality condition \eqref{finaleqz} states that there are constants $b_i$, $i=1,\dots,n$, and a $\mu_K$-measurable function $g(\omega)$ on $S^{n-1}$, such that
\begin{equation}\label{25}
V_1\left([o,\omega]|e_i^{\perp}\right)=b_ig(\omega),
\end{equation}
for $i=1,\dots,n$ and $\nu_K$-almost all $\omega\in S^{n-1}$.  Let $\omega_0=(\alpha_1,\dots,\alpha_n)$ and $\omega_1=(\beta_1,\dots,\beta_n)$ be two points in $S^{n-1}$ for which (\ref{25}) holds.  We may assume that
$g(\omega_1)\neq 0$, for otherwise $g(\omega)=0$ for all $\omega$ for which (\ref{25}) holds and this implies that $K=\{o\}$.
Let $c=g(\omega_0)/g(\omega_1)$.  Then from (\ref{25}) with $\omega=\omega_0$ and $\omega=\omega_1$, we obtain
$$\sum_{1\le k\le n,~k\neq i}\alpha_k^2=c^2\sum_{1\le k\le n,~k\neq i}\beta_k^2,$$
that is, $1-\alpha_i^2=c^2(1-\beta_i^2)$, for $i=1,\dots,n$.  Adding these equations gives $c^2=1$ and then we conclude that $\alpha_i=\pm \beta_i$, for $i=1,\dots,n$.  This shows that the support of $\mu_K$ must be a subset of the $2^n$ points in $S^{n-1}$ of the form $(\ee_1\alpha_1,\dots,\ee_n\alpha_n)$, where $\ee_i=\pm 1$, $i=1,\dots,n$.  Hence $K$ is a zonotope given by (\ref{zon1eq}).
\end{proof}

Inequality (\ref{eq1}) is not generally true for arbitrary convex bodies.  To see this, let $n=3$ and $m=1$, and recall that $C^3$ denotes the standard $o$-symmetric regular coordinate cross-polytope in $\R^3$.  Then $C^3|e_i^{\perp}$ is a square of side length $\sqrt{2}$, so $V_1(C^3|e_i^{\perp})=4\sqrt{2}/2=2\sqrt{2}$, for $i=1,2,3$.  Using (\ref{V1cross}), we see that (\ref{eq1}) is false for $C^3$ if and only if
$$\frac{\left(12\sqrt{2}\arccos(1/3)/2\pi\right)^2}
{3(2\sqrt{2})^2}< 1/2.$$
Computation shows that the left-hand side of the previous inequality is $0.46058...$, so (\ref{eq1}) is indeed false for $C^3$.

\section{The case $m<n-1$: Other results}\label{less}

The problem of finding a sharp inequality of the form (\ref{eq1}) that holds for general compact convex sets when $m<n-1$ appears to be difficult (see Problem~\ref{prob4}).  We can prove a weaker result, for which we need a definition.  Let $K$ be a convex body in $\R^n$ and let $m\in \{1,\dots,n-2\}$.  Then the $m$th area measure $S_m(K,\cdot)$ of $K$ satisfies the hypotheses of Minkowski's existence theorem (see, for example, \cite[(A.20), p.~399]{Gar06}) and hence is the surface area measure of a unique convex body $B_mK$ with centroid at the origin; in other words, $S(B_mK,\cdot)=S_m(K,\cdot)$.  In fact, it is enough to assume that $\dim K>m$ in order to conclude that a compact convex set $B_mK$ satisfying $S(B_mK,\cdot)=S_m(K,\cdot)$ exists. To see this, assume firstly that $K$ is a polytope, and suppose that the support of $S_m(K,\cdot)$ is contained in $u^\perp$ for some $u\in S^{n-1}$.  If $v\in S^{n-1}\setminus u^\perp$, the supporting set to $K$ in the direction $v$ must have dimension less than $m$. Hence the union of all such supporting sets also has dimension less than $m$. But this union contains the boundary of $K$ except the shadow boundary of $K$ in the direction $u$, and therefore has the same dimension as $K|u^\perp$.  However, if $\dim K>m$, then $\dim(K|u^\perp)\geq m$, a contradiction. We conclude that $S_m(K,\cdot)$ is not contained in $u^\perp$ for any $u\in S^{n-1}$, so it satisfies the hypotheses of Minkowski's existence theorem.  By approximation, the same conclusion is reached for arbitrary compact convex $K$ with $\dim K>m$.

The following lemma is stated with the assumption $\dim K>m$.  This is natural, since if $\dim K<m$, both sides of (\ref{eqmn}) are zero, while if $\dim K=m$, we have the equality provided by Proposition~\ref{CBm}.

\begin{thm}\label{mlemma}
Let $m\in \{1,\dots,n-2\}$ and $K$ be a compact convex set in $\R^n$ with $\dim K>m$.  Then
\begin{equation}\label{eqmn}
V_m(K)^2\ge \frac{1}{\pi}\left( \frac{\Gamma(\frac{n-m}{2})}{\Gamma(\frac{n-m+1}{2})}\right)^2 \sum_{i=1}^nV_{m}(K|e_i^{\perp})^2.
\end{equation}
\end{thm}

\begin{proof}
As was noted above, the assumption $\dim K>m$ guarantees that $B_mK$ exists.  By \cite[(A.34), p.~405]{Gar06}, we have
$$V_m(K)=\frac{\binom{n}{m}}{n\kappa_{n-m}}S_m(K,S^{n-1}).$$
This and the fact that $S_m(K,S^{n-1})=S(B_mK,S^{n-1})$ yields
\begin{equation}\label{f1}
V_m(K)=\frac{2\binom{n}{m}}{n\kappa_{n-m}}V_{n-1}(B_mK).
\end{equation}
By the generalized Cauchy projection formula \cite[(A.45), p.~408]{Gar06},
$$V_m(K|u^{\perp})=\frac{\binom{n-1}{m}}{2\kappa_{n-m-1}}
\int_{S^{n-1}}|u\cdot v|\,dS_m(K,v),$$
for all $u\in S^{n-1}$. This and $S_m(K,\cdot)=S(B_mK,\cdot)$ imply that
\begin{equation}\label{f2}
V_m(K|u^{\perp})=\frac{\binom{n-1}{m}}{\kappa_{n-m-1}}V_{n-1}(B_mK|u^{\perp}),
\end{equation}
for all $u\in S^{n-1}$.  Now using (\ref{f1}), (\ref{SQUARE}) with $K$ replaced by $B_mK$, and (\ref{f2}) with $u=e_i^{\perp}$, we obtain
\begin{eqnarray*}\label{roughbound}
V_m(K)^2 &= & \left( \frac{2\binom{n}{m}}{n\kappa_{n-m}}\right)^2 V_{n-1}(B_mK)^2\geq \left( \frac{2\binom{n}{m}}{n\kappa_{n-m}}\right)^2 \sum_{i=1}^nV_{n-1}(B_mK|e_i^{\perp})^2 \\
&= & \left( \frac{2\kappa_{n-m-1}\binom{n}{m}}{n\kappa_{n-m}\binom{n-1}{m}}\right)^2 \sum_{i=1}^nV_{m}(K|e_i^{\perp})^2 = \frac{1}{\pi}\left( \frac{\Gamma(\frac{n-m}{2})}{\Gamma(\frac{n-m+1}{2})}\right)^2 \sum_{i=1}^nV_{m}(K|e_i^{\perp})^2.
\end{eqnarray*}
\end{proof}

The previous bound is not optimal.  Indeed, the proof of Theorem~\ref{mlemma} shows that equality in (\ref{eqmn}) would imply that equality holds in the left-hand inequality in (\ref{SQUARE}) when $K$ is replaced by $B_mK$.  The equality condition for (\ref{SQUARE}) then yields that either $\dim B_mK\le n-1$ or $B_mK$ is a coordinate cross-polytope.  In either case, the surface area measure of $B_mK$, which is just the $m$th area measure of $K$, would have atoms unless it is the zero measure.  This contradicts \cite[Theorem~4.6.5]{Sch93}, which states that an $m$th area measure cannot be positive on sets whose Hausdorff dimension is less than $n-m-1$.

For example, when $n=3$ and $m=1$, the constant in (\ref{eqmn}) is
$$\frac{1}{\pi}\left( \frac{\Gamma(1)}{\Gamma(\frac{3}{2})}\right)^2=\frac{4}{\pi^2}=0.40528...,$$
lower than the probable best bound $0.46058...$ for the regular coordinate cross-polytope.  (Note that it is higher than constant $1/3$ in the easy bound (\ref{easynbound}).)

The hypothesis (\ref{CGconj}) in the following lemma was shown in \cite[Theorem~4.1]{CamG11} to be equivalent to the existence of a coordinate box $Z$ such that $V_m(Z|e_i^{\perp})=V_m(K|e_i^{\perp})$, for $i=1,\dots,n$.  Inequality (\ref{CGconj}) is true when $m=1$ or $m=n-1$; this follows from (\ref{trivmax2}) and \cite[Theorem~3.1]{CamG11} or the left-hand inequality in Theorem~\ref{generalBetkeMcMullen} with $a_i=1$ for $i=1,\dots,n$, respectively.  We prove below in Lemma~\ref{m2} that (\ref{CGconj}) is also true when $m=n-2$.  For $m\in \{2,\dots,n-3\}$, (\ref{CGconj}) remains a conjecture.

\begin{lem}\label{lemnewer}
Let $K$ be a compact convex set in $\R^n$ and let $m\in \{1,\dots,n-2\}$.  If
\begin{equation}\label{CGconj}
V_m(K|e_k^{\perp})\le \frac{1}{n-m}\sum_{i=1}^nV_m(K|e_i^{\perp}),
\end{equation}
for $k=1,\dots,n$, then
\begin{equation}\label{eq111}
\sum_{i=1}^nV_m(K|e_i^{\perp})^2\le \frac{1}{n-m}\left(\sum_{i=1}^nV_m(K|e_i^{\perp})\right)^2.
\end{equation}
\end{lem}

\begin{proof}
Let $V_m(K|e_i^{\perp})=c_i$, for $i=1,\dots,n$.  By homogeneity, we may assume without loss of generality that $\sum_{i=1}^nc_i=1$.  Using (\ref{CGconj}), we obtain the additional constraints $0\le c_k\le 1/(n-m)$, for $k=1,\dots,n$.

The set of $(c_1,\dots,c_n)$ satisfying these constraints is an $(n-1)$-dimensional convex polytope $P$ contained in the hyperplane $\{x=(x_1,\dots,x_n)\in\R^n:x_1+\cdots+x_n=1\}$.  The maximum distance $d$ from the origin to a point in $P$ is attained at a vertex of $P$.  At such a vertex, we have either $c_k=0$ or $c_k=1/(n-m)$ for at least $n-1$ of the $k$'s.  If less than $m$ of these $c_k$'s are zero, we would have $\sum_{k=1}^nc_k>1$, a contradiction.  Therefore at least $m$ of the $c_k$'s are zero.  Consequently,
\begin{eqnarray*}
d&=&\left(\sum_{i=1}^nc_i^2\right)^{1/2}\le \left((n-m)\max_{1\le k\le n}c_k^2\right)^{1/2}\\
&\le &\left((n-m)\frac{1}{(n-m)^2}\right)^{1/2}
=\frac{1}{\sqrt{n-m}}=\frac{1}{\sqrt{n-m}}\sum_{i=1}^nc_i,
\end{eqnarray*}
which yields (\ref{eq111}).
\end{proof}

\begin{lem}\label{m2}
Let $K$ be a compact convex set in $\R^n$.  Then (\ref{CGconj}) holds when $m=n-2$.
\end{lem}

\begin{proof}
Without loss of generality, let $k=1$.  We identify $e_1^{\perp}$ with $\R^{n-1}$ and apply the left-hand inequality in Proposition~\ref{generalBetkeMcMullen} with $a_i=1$ for each $i$ and $K$ and $n$ replaced by $L=K|e_1^\perp$ and $n-1$, respectively, to obtain
$$
    V_{n-2}(L)\leq \sum_{i=2}^n V_{n-2}(L|e_i^\perp).
$$
By (\ref{cgtriv}) with $m$, $K$, and $u$ replaced by $n-2$, $K|e_i^\perp$, and $e_1$, respectively, we have
$$V_{n-2}(L|e_i^\perp)= V_{n-2}\left((K|e_1^\perp)|e_i^\perp\right)
=V_{n-2}\left((K|e_i^\perp)|e_1^\perp\right)\leq V_{n-2}(K|e_i^\perp),$$ for
$i=2,\dots,n$.  It follows that
$$V_{n-2}(K|e_1^\perp) = V_{n-2}(L) \leq
     \sum_{i=2}^n V_{n-2}(L|e_i^\perp)\leq \sum_{i=2}^n V_{n-2}(K|e_i^\perp).$$
Adding $V_{n-2}(K|e_1^\perp)$ to both sides, we obtain (\ref{CGconj}) with $m=n-2$.
\end{proof}

When $m=1$ or $m=n-2$, the following result establishes a relationship between the lower bound for $V_m(K)$ for a zonoid $K$ from (\ref{eq1}) and the upper bound for $V_m(K)$ from \cite[Theorem~3.1]{CamG11} (for $m=1$) or Lemma~\ref{m2} (for $m=n-2$). It represents a reverse Cauchy-Schwarz inequality for the numbers $V_m(K|e_i^{\perp})$, $i=1,\dots,n$.  We do not know if the result holds for $m\in \{2,\dots,n-3\}$; see Problem~\ref{prob3}.

\begin{thm}\label{lemnew}
Let $K$ be a compact convex set in $\R^n$ and let $m=1$ or $m=n-2$. Then
\begin{equation}\label{eq11}
\sum_{i=1}^nV_m(K|e_i^{\perp})^2\le \frac{1}{\sqrt{n-m}}\left(\sum_{i=1}^nV_m(K|e_i^{\perp})\right)^2.
\end{equation}
\end{thm}

\begin{proof}
When $m=1$, (\ref{CGconj}) holds, by (\ref{trivmax2}) for $m=1$ and
\cite[Theorem~3.1]{CamG11}, and when $m=n-2$, (\ref{CGconj}) holds by Lemma~\ref{m2}. Then (\ref{eq11}) with $m=1$ or $m=n-2$ follows directly from Lemma~\ref{lemnewer}.
\end{proof}

We end this section with a counterpart to \cite[Theorem~4.1]{CamG11}.

\begin{thm}\label{SegmentExistence}
Let $K$ be a compact convex set in $\R^n$. There is a line segment $L$ such that
\begin{equation}\label{eq2}
V_1(L|e_i^{\perp})=V_1(K|e_i^{\perp}),
\end{equation}
for $i=1,\dots,n$, if and only if
\begin{equation}\label{eq2a}
V_1(K|e_i^{\perp})^2\le \frac{1}{n-1}\sum_{k=1}^nV_1(K|e_k^{\perp})^2,
\end{equation}
for $i=1,\dots,n$.
\end{thm}

\begin{proof}
Let $V_1(K|e_i^\perp)=a_i$, for $i=1,\dots,n$.  If $L=[-x/2,x/2]$ and $x=(x_1,\dots,x_n)$, then (\ref{eq2}) is equivalent to
\begin{equation}\label{SystemV1}
\sum_{1\le k\le n,~k\neq i}x_k^2=a_i^2,
\end{equation}
for $i=1,\dots,n$.  Summing the equations in (\ref{SystemV1}) over $i$, we obtain
$$
(n-1)\sum_{i=1}^nx_i^2=\sum_{i=1}^na_i^2.
$$
Subtracting $(n-1)$ times the $i$th equation in (\ref{SystemV1}), we get
\begin{equation}\label{SystemV2}
(n-1)x_i^2 = \sum_{k=1}^na_k^2 -(n-1)a_i^2.
\end{equation}
The left-hand side is nonnegative, so the right-hand side is also, and this is equivalent to (\ref{eq2a}).

Conversely, assuming that (\ref{eq2a}) holds, we know that the right-hand side of (\ref{SystemV2}) is nonnegative and hence (\ref{SystemV2}) can be solved for $x_i$.  The values of $x_i$ thus obtained also satisfy (\ref{SystemV1}), so (\ref{eq2}) holds for $L=[-x/2,x/2]$ when $x=(x_1,\dots,x_n)$.
\end{proof}

There are convex bodies for which (\ref{eq2a}) does not hold.  To see this, let $P$ be the coordinate box defined by $P=\sum_{k=1}^{n}s_k[-e_k/2,e_k/2]$. Suppose that $s_k>0$ for $k=1,\dots,n-1$ and $s_n=0$. If we take $i=n$ and $K=P$, then (\ref{eq2a}) becomes
$$
(n-1)\left( \sum_{k=1}^{n-1}s_k\right) ^2\leq \sum_{l=1}^{n}\left( \sum_{ 1\le k\le n-1,~k\neq l}s_k\right)^2,
$$
which yields
$$
\sum_{1\le k\le n-1,~k\neq l}s_ks_l\leq 0.
$$
This is false, so (\ref{eq2a}) does not hold for $P$ under the assumptions above.  By continuity, it is also false for the $n$-dimensional coordinate box $P=\sum_{k=1}^{n}s_k[-e_k/2,e_k/2]$ when $s_n>0$ is sufficiently small.

The fact that an $(n-1)$-dimensional coordinate box such as $P$ does not satisfy (\ref{eq2a}) can also be seen as follows. Since $s_n=0$, we have $V_1(P)=V_1(P|e_n^\perp)$.  By Corollary~\ref{Zonoidmcor}, we know that strict inequality holds in (\ref{eq1}) when $K=P$.  This contradicts
(\ref{eq2a}) when $K=P$ and $i=n$.

Note that if $K$ does satisfy (\ref{eq2a}), Theorem~\ref{SegmentExistence} states that there is a line segment $L\in {\mathcal{P}}(K,1)$.  Since equality holds in Theorem~\ref{Zonoidm} when $K=L$, we see by that theorem that $L$ is a zonoid of minimal mean width in ${\mathcal{P}}(K,1)$.

\section{Open problems}\label{problems}

\begin{prob}\label{prob3}
{\em Does (\ref{eq11}) hold when $m\in \{2,\dots,n-3\}$?}
\end{prob}

\begin{prob}\label{prob2}
{\em Is there a version of Theorem~\ref{SegmentExistence} for $m\in \{2,\dots,n-2\}$?}
\end{prob}

\begin{prob}\label{prob4}
{\em Let $K$ be a compact convex set in $\R^n$ and let $m\in \{1,\dots,n-2\}$.  Then is there a constant $c_2=c_2(n,m)$ such that
$$
V_m(K)^2\geq c_2\sum_{i=1}^nV_m(K|e_i^{\perp})^2\ge c_2\sum_{i=1}^nV_m(K\cap e_i^{\perp})^2,
$$
with equality in either inequality involving $V_m(K)$ when $\dim K=n$ if and only if $K$ is an $o$-symmetric regular coordinate cross-polytope (or one of its translates, in the case of the left-hand inequality)?}
\end{prob}

\begin{prob}\label{prob5}
{\em Let $K$ be a convex body in $\R^n$ and let $m\in \{1,\dots,n-2\}$.  Then is there a constant $c_3=c_3(n,m)$ such that
$$
V_{m+1}(K)^{mn}\geq c_3\prod_{i=1}^nV_m(K\cap e_i^{\perp})^{m+1},
$$
with equality if and only if $K$ is an $o$-symmetric coordinate cross-polytope?}
\end{prob}

\begin{prob}\label{prob5a}
{\em Let $K$ be a convex body in $\R^3$.  Is
$$
V_{2}(K)^2\geq \frac{1}{16}\left(\sum_{i=1}^3V_1(K\cap e_i^{\perp})^2\right)^2
-\frac{1}{8}\sum_{i=1}^3V_1(K\cap e_i^{\perp})^4,$$
with equality if and only if $K$ is an $o$-symmetric coordinate cross-polytope?}
\end{prob}

An upper bound for $V_2(K)$ analogous to the lower bound in the previous problem was obtained in \cite[Theorem~4.6]{CamG11}.  The proposed lower bound clearly relates to Heron's formula; in one version, this states that a triangle with sides of length $a$, $b$, and $c$ has area
$$\frac{1}{4}\sqrt{(a^2+b^2+c^2)^2-2(a^4+b^4+c^4)}.$$

\section*{Appendix: Proof of the case $n\ge 4$ of Lemma~\ref{lembm3}}\label{Appendix}

This appendix is devoted to proving that if $K$ is as in Lemma~\ref{leman} and $n\ge 4$, then the function $J_K$ defined by \eqref{jj} is increasing.  It has already been observed in the proof of Lemma~\ref{lembm3} that since $K$ has the same symmetries as $Q^n$, we have that for each $i\in \{1,\dots,n-1\}$, $h_K$ is convex and even as a function of $x_i$ and hence increases with $x_i$ for $x_i\in [0,1]$.  We recall that by (\ref{x1eq}),
\begin{equation}\label{x1eq2}
x_1^2+\dots+x_{n-1}^2=1.
\end{equation}

The plan is to consider the cases $n=4$ and $n=5$ separately and then dispose of the remaining case $n\ge 6$ by means of an induction argument.

Let $n=4$.  By (\ref{jj}),
$$J_K(x_3)=\frac{\int_{x_3}^{\sqrt{\frac{1-x_3^2}{2}}}h_K\left(1,\frac{x_2}{x_1},\frac{x_3}
{x_1},0\right)\,dx_2}
{\sqrt{\frac{1-x_3^2}{2}}-x_3}=\int_{0}^{1}h_K\left(1,\frac{x_2}{x_1},\frac{x_3}
{x_1},0\right)\,dt,
$$
where
\begin{equation}\label{subsx2}
x_2=x_3+t\left(\sqrt{\frac{1-x_3^2}{2}}-x_3\right)
\end{equation}
and where $x_2$ and hence $x_1$ are now functions of $x_3$ and $t$.  It therefore suffices to show that $x_2/x_1$ and $x_3/x_1$ are increasing functions of $x_3$ for any fixed $t$.  To this end, using (\ref{x1eq2}), (\ref{subsx2}), and straightforward but quite tedious calculation, we find that
$$x_1^3\frac{\partial (x_2/x_1)}{\partial x_3}=1-t\ge 0,$$
showing that $x_2/x_1$ increases with $x_3$ for fixed $t$.  Similarly, further calculations yield
$$x_1^3\frac{\partial (x_3/x_1)}{\partial x_3}=
\frac{(2-t^2)\sqrt{(1-x_3^2)/2}-t(1-t)x_3}{\sqrt{2}\sqrt{1-x_3^2}}\ge
\frac{(2-t)x_3}{\sqrt{2}\sqrt{1-x_3^2}}\ge 0,$$
where we used $\sqrt{(1-x_3^2)/2}\ge x_3$.  This shows that $x_3/x_1$ increases with $x_3$ for fixed $t$ and completes the proof for $n=4$.

Henceforth we assume that $n\ge 5$ and make a change of variables by setting
\begin{equation}\label{yieq}
y_i=x_i/x_1,~ i=2,\dots,n-2,~\quad~{\text{and}}~\quad~y_{n-1}=x_{n-1}.
\end{equation}
Then
$$
\frac{\partial y_i}{\partial x_j}=\begin{cases}
(x_1^2+x_i^2)/x_1^3,& {\text{if $i=j=2,\dots,n-2$,}}\\[1ex]
x_ix_j/x_1^3,& {\text{if $i, j=2,\dots, n-2$, $j\neq i$}}.
\end{cases}
$$
The Jacobian of this transformation may be calculated by first noticing that the determinant reduces to one of size $(n-3)\times (n-3)$, since apart from the $(n-2,n-2)$ entry, the last row and column have all entries equal to zero. Then, factor $x_i/x_1^3$ from the $i$th column and $x_j$ from the $j$th row.  Next, subtract the first column from all the others and multiply the $i$th row by $x_{i+1}^2$. Finally, add each row to the first row.  The result is $1/x_1^{3(n-3)}$ times the determinant of a lower triangular matrix with first diagonal entry $\sum_{i=1}^{n-2}x_i^2$ and all other diagonal entries equal to $x_1^2$. Therefore, using (\ref{x1eq2}), the Jacobian reduces to
$$
\frac{\partial (y_2,\dots, y_{n-1})}{\partial (x_2,\dots, x_{n-1})} = \frac{1-y_{n-1}^2}{x_1^{n-1}}.$$
It will be convenient to set, for $i=2,\dots,n-2$,
\begin{equation}\label{Yeq}
    Y_i=(1, y_2, \dots, y_{i})\quad\Rightarrow\quad
  |Y_i|^2=1+y_2^2+\cdots+y_{i}^2.
\end{equation}
Then, using (\ref{x1eq2}) and (\ref{yieq}) to get
\begin{equation}\label{x1eq3}
x_1^2=\frac{1-y_{n-1}^2}{1+y_2^2+\cdots+y_{n-2}^2}=\frac{1-y_{n-1}^2}{|Y_{n-2}|^2},
\end{equation}
we can rewrite the Jacobian as
\begin{equation}\label{Jaceq}
\frac{\partial (y_2,\dots, y_{n-1})}{\partial (x_2,\dots, x_{n-1})} =
\frac{(1+y_2^2+\dots+y_{n-2}^2)^{(n-1)/2}} {(1-y_{n-1}^2)^{(n-3)/2}}=
\frac{|Y_{n-2}|^{n-1}} {(1-y_{n-1}^2)^{(n-3)/2}}.
\end{equation}

In order to describe the region of integration in the expression (\ref{jj}) for $J_K$, we begin by recalling that the general region $\Omega\subset S^{n-1}$ of interest is given by (\ref{omdef}) and hence
\begin{equation}\label{omdef33}
0\le x_{n-1}\le x_{n-2}\le\cdots \le x_1.
\end{equation}
We already know from (\ref{inarray}) and (\ref{yieq}) that
$$0\leq y_{n-1}\leq \frac{1}{\sqrt{n-1}}.$$
To bound $y_2=x_2/x_1$, observe that by (\ref{x1eq2}), $y_2$ is an increasing function of $x_2$. Therefore by (\ref{yieq}) and (\ref{omdef33}), the maximum and minimum of $y_2$ occur when $x_2=x_1$ and $x_2=x_3=\cdots =x_{n-2}=y_{n-1}$, respectively. By (\ref{x1eq2}), this gives $L_2\le y_2\le 1$, where
\begin{equation}\label{L2eq}
L_2=\frac{y_{n-1}}{(1-(n-2)y_{n-1}^2)^{1/2}}.
\end{equation}
Similarly, if $i\in\{3,\dots,n-3\}$, once $y_{n-1}, y_{2},\dots,y_{i-1}$ are fixed, the maximum and minimum
of $y_i=x_i/x_1$ occur when $x_i=x_{i-1}$ and $x_{i}=x_{i+1}=\cdots=x_{n-2}=y_{n-1}$, respectively.  Using (\ref{x1eq2}) and (\ref{yieq}) again, we find that $L_i\le y_i\le y_{i-1}$, where
\begin{equation}\label{Lieq}
L_i=\frac{y_{n-1}(1+y_2^2+\dots+y_{i-1}^2)^{1/2}}{(1-(n-i)y_{n-1}^2)^{1/2}}
=\frac{y_{n-1}|Y_{i-1}|}{(1-(n-i)y_{n-1}^2)^{1/2}},
\end{equation}
for $i=3,\dots, n-2$.  Consequently, (\ref{Yeq}), (\ref{x1eq3}), (\ref{Jaceq}), (\ref{L2eq}), and (\ref{Lieq}) allow the function $J_K$ defined by \eqref{jj} to be rewritten as
\begin{equation}\label{JJJ}
J_K(y_{n-1})=\frac{\int_{L_2}^{1} \int_{L_3}^{y_2}\cdots \int_{L_{n-2}}^{y_{n-3}}
h_K\left(Y_{n-2}, \frac{y_{n-1}|Y_{n-2}|}{\sqrt{1-y_{n-1}^2}},0\right) |Y_{n-2}|^{1-n} \,dy_{n-2}\cdots dy_2}{\int_{L_2}^{1} \int_{L_3}^{y_2}\cdots \int_{L_{n-2}}^{y_{n-3}}|Y_{n-2}|^{1-n}\,dy_{n-2}\cdots dy_2}.
\end{equation}
(Note that the denominator of (\ref{Jaceq}) does not depend on $y_2,\dots,y_{n-2}$ and so can be factored from the integrals in the numerator and denominator of $J_K$ and then canceled.)  Here, and in what follows, we abbreviate $h_K(1,y_2,\dots,y_{i},z_{i+1},
\dots,z_{n-1},0)$, for $i=2,\dots,n-2$, by writing $h_K(Y_{i},z_{i+1}
\dots,z_{n-1},0)$ instead.

If $h_K$ is differentiable, then by its definition, $J_K$ is also differentiable with respect to $y_{n-1}$ on $(0,1/\sqrt{n-1})$. Assuming this, we shall prove that the derivative is nonnegative and hence that $J_K$ is increasing.  In fact, we may assume without loss of generality that $h_K$ is differentiable, or, equivalently (see \cite[p.~107]{Sch93}) that $K$ is strictly convex.  Indeed, if this is not the case, we may choose a sequence $\{K_i\}$ of strictly convex bodies converging to $K$ in the Hausdorff metric; see \cite[p.~158--160]{Sch93} for even stronger results of this type.  Then $h_{K_i}$ converges uniformly on $S^{n-1}$ to $h_K$ (see \cite[p.~54]{Sch93}) and hence $J_{K_i}$ converges to $J_K$.  Therefore, if each $J_{K_i}$ is increasing, $J_K$ is also increasing.

We proceed to differentiate $J_K$ with respect to $y_{n-1}$.  We shall use the fact that for $i=2,\dots,n-3$,
\begin{equation}\label{Liobs}
y_i=L_i\quad\Rightarrow \quad L_{i+1}=L_i.
\end{equation}
Indeed, (\ref{L2eq}) and (\ref{Lieq}) imply that both equations in (\ref{Liobs}) are equivalent to
$$
  1+y_2^2+\dots +y_i^2=(1+y_2^2+\dots +y_{i-1}^2)\,  \frac{1-(n-i-1)y_{n-1}^2}{1-(n-i)y_{n-1}^2}.
$$
Let $J_K=N/D$, where $N=N(K)$ and $D$ are the numerator and denominator in (\ref{JJJ}), and let $z=y_{n-1}|Y_{n-2}|/\sqrt{1-y_{n-1}^2}$.  Applying Leibniz's rule for differentiating the integrals $N$ and $D$, we notice that the terms involving the derivatives of the limits $L_2$ and $1$ of the integrals with respect to $y_2$ vanish, since $y_2=L_2$ implies $L_2=L_3$, in view of (\ref{Liobs}), and hence $y_2=L_3$.  Similarly, for $i=3,\dots,n-3$, terms involving the derivatives of the limits $L_i$ and $y_{i-1}$ of the integrals with respect to $y_i$ vanish, since (\ref{Liobs}) says that $y_i=L_i$ implies $L_i=L_{i+1}$, and then $y_i=L_{i+1}$.  Consequently, $(dJ_K/dy_{n-1})D^2$ equals
\begin{align*}
\MoveEqLeft{D\int\limits_{L_2}^{1} \int\limits_{L_3}^{y_2}\cdots \int\limits_{L_{n-3}}^{y_{n-4}}
   \left( -\frac{\partial L_{n-2}}{\partial y_{n-1}}
\left.\left(h_K(Y_{n-2}, z,0)|Y_{n-2}|^{1-n} \right)\right|_{y_{n-2}=L_{n-2}} +  \int_{L_{n-2}}^{y_{n-3}}\frac{e_{n-1}\cdot\nabla h_K(Y_{n-2}, z,0)}{(1-y_{n-1}^2)^{3/2}|Y_{n-2}|^{n-2}}\right)}\\
&\times dy_{n-3}\cdots dy_2 -N\int_{L_2}^{1} \int_{L_3}^{y_2}\cdots \int_{L_{n-3}}^{y_{n-4}}
\left(-\frac{\partial L_{n-2}}{\partial y_{n-1}}
\left.|Y_{n-2}|^{1-n} \right|_{y_{n-2}=L_{n-2}}\right) \, dy_{n-3}\cdots dy_2.
\end{align*}
Recalling that all components of $\nabla h_K$ are nonnegative, we conclude that the second term in the first integral is nonnegative.  Substituting
$$
  \frac{\partial L_{n-2}}{\partial y_{n-1}}= \frac{(1+y_2^2+\dots +y_{n-3}^2)^{1/2}}{(1-2y_{n-1}^2)^{3/2}}= \frac{|Y_{n-3}|}{(1-2y_{n-1}^2)^{3/2}},
$$
we find that $dJ_K/dy_{n-1}$ is at least a positive constant multiple of
\begin{align*}
\MoveEqLeft{-D\int_{L_2}^{1} \int_{L_3}^{y_2}\cdots \int_{L_{n-3}}^{y_{n-4}}
    \frac{(1-2y_{n-1}^2)^\frac{n-4}{2}}{|Y_{n-3}|^{n-2}(1-y_{n-1}^2)^\frac{n-1}{2}}
 \left.h_K(Y_{n-2},z,0)\right|_{y_{n-2}=L_{n-2}}\, dy_{n-3}\cdots dy_2}\\
&+N\int_{L_2}^{1} \int_{L_3}^{y_2}\cdots \int_{L_{n-3}}^{y_{n-4}}
 \frac{(1-2y_{n-1}^2)^\frac{n-4}{2}}{|Y_{n-3}|^{n-2}(1-y_{n-1}^2)^\frac{n-1}{2}} \, dy_{n-3}\cdots dy_2.
\end{align*}
The common expressions depending only on $y_{n-1}$ can be factored and absorbed into the constant multiplying factor. Furthermore, the restriction $y_{n-2}\in [L_{n-2}, y_{n-3}]$ and the fact that $h_K(Y_{n-2},z,0)$ is increasing with respect to $y_{n-2}$ means that $h_K(Y_{n-2},z,0)$ has its minimum when $y_{n-2}=L_{n-2}$.  Letting
$$
R=R(y_2, \dots, y_{n-3}, y_{n-1}) = \int_{L_{n-2}}^{y_{n-3}}|Y_{n-2}|^{1-n} \,dy_{n-2}=\int_{L_{n-2}}^{y_{n-3}}\left(|Y_{n-3}|^2+t^2\right)^{(1-n)/2} \,dt
$$
and using the expressions for $N$ and $D$ from (\ref{JJJ}), we see that $dJ_K/dy_{n-1}$ is at least a positive constant multiple of
\begin{align*}
\MoveEqLeft{-\left(\int_{\Sigma_1}R \,dx\right)\left( \int_{\Sigma_1} |Y_{n-3}|^{2-n}
 \left.h_K(Y_{n-2}, z,0)\right|_{y_{n-2}=L_{n-2}}\, dx\right)}\\
& +\left(\int_{\Sigma_1}
R\left.h_K(Y_{n-2},z,0)\right|_{y_{n-2}=L_{n-2}}\,dx\right)
\left(\int_{\Sigma_1}|Y_{n-3}|^{2-n} \, dx\right),
\end{align*}
where $dx=dy_{n-3}\cdots dy_2$ and $\Sigma_1$ is the corresponding region of integration from the previous integrals.
Note that when $y_{n-2}=L_{n-2}$, (\ref{Lieq}) with $i=n-2$ implies that $z=L_{n-2}$, so $h_K(Y_{n-2},z,0)=h_K(Y_{n-3}, L_{n-2}, L_{n-2},0)$. To prove that $dJ_K/dy_{n-1}\ge 0$, it will therefore suffice to show that
\begin{equation}\label{Equivalent1}
\int_{\Sigma_1}h_K(Y_{n-3}, L_{n-2}, L_{n-2},0) \left(\frac{\int\limits_{L_{n-2}}^{y_{n-3}}\left(|Y_{n-3}|^2+t^2\right)
^\frac{1-n}{2} \,dt}{\int\limits_{\Sigma_1}\int\limits_{L_{n-2}}^{y_{n-3}}
\left(|Y_{n-3}|^2+t^2\right)^\frac{1-n}{2} \,dt\,dx}
  -\frac{|Y_{n-3}|^{2-n}}{\int\limits_{\Sigma_1}
  |Y_{n-3}|^{2-n}\,dx} \right)\,dx \geq 0.
\end{equation}
The substitution $t=|Y_{n-3}|u$ allows (\ref{Equivalent1}) to be rewritten in the form
\begin{equation}\label{n5case}
\int_{\Sigma_1}h_K(Y_{n-3}, L_{n-2}, L_{n-2},0) \left(\frac{|Y_{n-3}|^{2-n}\,U}{\int_{\Sigma_1}|Y_{n-3}|^{2-n}\,U\,dx}
  -\frac{|Y_{n-3}|^{2-n}}{\int_{\Sigma_1}|Y_{n-3}|^{2-n}\,dx}\right)\,dx \geq 0,
\end{equation}
where
\begin{equation}\label{Ueqn}
U=\int_{y_{n-1}/(1-2y_{n-1}^2)^{1/2}}
^{y_{n-3}/|Y_{n-3}|}
(1+u^2)^\frac{1-n}{2} \,du.
\end{equation}

Let $n=5$.  Then $\Sigma_1=[L_2,1]=[y_4/\sqrt{1-3y_4^2},1]$ by (\ref{L2eq}), and $dx=dy_2$.  In view of (\ref{Lieq}) with $i=3$ and the fact that $h_K$ increases with respect to its arguments, $h_K(Y_{2}, L_{3}, L_{3},0)$ is increasing with respect to $y_2$.  The part of the integrand in (\ref{n5case}) in parentheses, $S$ say, is clearly continuous with zero average over $\Sigma_1$.  The factor $|Y_{n-3}|^{2-n}=(1+y_2^2)^{-3/2}$ in $S$ is nonnegative and the remaining factor is increasing with respect to $y_2$, since only the upper limit $y_2/\sqrt{1+y_2^2}$ in the integral expression for $U$ depends on $y_2$.  Therefore there is some $c=c(y_4)\in [y_4/\sqrt{1-3y_4^2},1]$ such that $S\le 0$ for $y_2\in [y_4/\sqrt{1-3y_4^2},c]$ and $S\ge 0$ for $y_2\in [c,1]$.  Applying Lemma~\ref{Cheblem} with $f=S$ and $g=h_K$, we obtain (\ref{n5case}).  This completes the proof for $n=5$.

For the remainder of the proof we assume that $n\ge 6$.  The proof will be by induction on $n$.  Assume that the lemma holds for all dimensions less than $n$.  We shall make two further changes of variables, the first of which is to let $v_{n-3}=y_{n-3}/|Y_{n-3}|$. Then it is easy to check that
\begin{equation}\label{Jac2}
\frac{\partial (y_2,\dots, y_{n-4}, v_{n-3})}{\partial (y_2,\dots, y_{n-3})} = \frac{1+y_2^2+\dots+y_{n-4}^2}{(1+y_2^2+\dots+y_{n-3}^2)^{3/2}}
=\frac{|Y_{n-4}|^2}{|Y_{n-3}|^3}.
\end{equation}
We also have
\begin{equation}\label{yn3eq}
  y_{n-3}=\frac{v_{n-3}|Y_{n-4}|}{\sqrt{1-v_{n-3}^2}}
\end{equation}
and
\begin{equation}\label{Ln2eq}
 L_{n-2}=\frac{y_{n-1}|Y_{n-4}|}
 {\sqrt{1-2y_{n-1}^2}\ \sqrt{1-v_{n-3}^2}}.
\end{equation}
Setting
\begin{equation}\label{Weqn}
  T=T(v_{n-3},y_{n-1})=\frac{U}{\int_{\Sigma_1}|Y_{n-3}|^{2-n}\,U\,dx}
  -\frac{1}{\int_{\Sigma_1}|Y_{n-3}|^{2-n}\,dx}
\end{equation}
and noting that by (\ref{yn3eq}) and (\ref{Ln2eq}), we have
$$h_K=h_K\left(1, y_2, \dots, y_{n-4}, \frac{v_{n-3}|Y_{n-4}|}{\sqrt{1-v_{n-3}^2}}, \frac{y_{n-1}|Y_{n-4}|}
 {\sqrt{1-2y_{n-1}^2}\ \sqrt{1-v_{n-3}^2}}, \frac{y_{n-1}|Y_{n-4}|}
 {\sqrt{1-2y_{n-1}^2}\ \sqrt{1-v_{n-3}^2}},0\right),$$
we use (\ref{Jac2}) to rewrite (\ref{Equivalent1}) in the form
\begin{equation}\label{Equivalent2}
\int_{\Sigma_2}h_K|Y_{n-3}|^{2-n}T\frac{|Y_{n-3}|^3}
{|Y_{n-4}|^2}\,dx=
\int_{\Sigma_2}h_K\frac{(1-v_{n-3}^2)^\frac{n-5}{2}}{|Y_{n-4}|^{n-3}}T
\,dx\geq 0.
\end{equation}
Here $\Sigma_2$ is the new domain of integration obtained from $\Sigma$ by the last change of variable, given explicitly by
$$
\int_{\Sigma_2}\,dx= \int_{L_2}^{1}\int_{L_3}^{y_2}\cdots \int_{L_{n-4}}^{y_{n-5}}\int_{\frac{y_{n-1}}{\sqrt{1-2y_{n-1}^2}}} ^{\frac{y_{n-4}}{\sqrt{|Y_{n-4}|^2+y_{n-4}^2}}}dv_{n-3}\,dy_{n-4}\cdots dy_2,
$$
where we used (\ref{Lieq}) with $i=n-3$ to obtain the lower limit for integration with respect to $v_{n-3}$.  The sign of the integrand in (\ref{Equivalent2}) coincides with the sign of $T$.  Also, by (\ref{Ueqn}) and (\ref{Weqn}), $T$ is increasing with respect to $v_{n-3}$.  It follows that the sign of the integrand in (\ref{Equivalent2}) coincides with that of $v_{n-3}-m(y_{n-1})$, for a suitable function $m$ of $y_{n-1}$. Hence, since $h_K$ is also increasing with respect to $v_{n-3}$, the integral in (\ref{Equivalent2}) is
\begin{align}\label{xyz}
\MoveEqLeft{\int_{L_2}^{1} \int_{L_3}^{y_2}\cdots \int_{L_{n-4}}^{y_{n-5}}\left(-\int_{\frac{y_{n-1}}{\sqrt{1-2y_{n-1}^2}}} ^{m(y_{n-1})} h_K\frac{(1-v_{n-3}^2)^\frac{n-5}{2}}{|Y_{n-4}|^{n-3}}|T| \, dv_{n-3}\right.}\nonumber\\
& + \left.\int_{m(y_{n-1})} ^{\frac{y_{n-4}}{\sqrt{|Y_{n-4}|^2+y_{n-4}^2}}} h_K\frac{(1-v_{n-3}^2)^\frac{n-5}{2}}{|Y_{n-4}|^{n-3}}T \, dv_{n-3}\right)\, dy_{n-4}\cdots dy_2\nonumber\\
& \geq \int_{L_2}^{1} \int_{L_3}^{y_2}\cdots \int_{L_{n-4}}^{y_{n-5}}h_K(M)\left(\int_{\frac{y_{n-1}}{\sqrt{1-2y_{n-1}^2}}} ^{\frac{y_{n-4}}{\sqrt{|Y_{n-4}|^2+y_{n-4}^2}}}  \frac{(1-v_{n-3}^2)^\frac{n-5}{2}}{|Y_{n-4}|^{n-3}}T\, dv_{n-3}\right)\, dy_{n-4}\cdots dy_2,
\end{align}
where
$$M=\left(1, y_2, \dots, y_{n-4}, \frac{m|Y_{n-4}|}{\sqrt{1-m^2}}, \frac{y_{n-1}|Y_{n-4}|}
{\sqrt{1-2y_{n-1}^2}\ \sqrt{1-m^2}}, \frac{y_{n-1}|Y_{n-4}|}
{\sqrt{1-2y_{n-1}^2}\ \sqrt{1-m^2}},0\right).$$

Our aim is to show that the previous integral is nonnegative.  To this end, we introduce our second and final change of variables, by letting $v_i=y_i/|Y_{n-4}|$, for $i=2,\dots,n-4$. Then
$$
\frac{\partial v_i}{\partial y_j}=\begin{cases}
(|Y_{n-4}|^2-y_i^2)/|Y_{n-4}|^3,& {\text{if $i=j=2,\dots,n-4$,}}\\[1ex]
-y_iy_j/|Y_{n-4}|^3,& {\text{if $i, j=2,\dots, n-4$, $j\neq i$}},
\end{cases}
$$
and by manipulations similar to those set out for the initial change of variables (\ref{yieq}), we find that
$$
\frac{\partial (v_2,\dots,v_{n-4})}{\partial (y_2,\dots,y_{n-4})} = |Y_{n-4}|^{3-n}.
$$
It is easy to check that in (\ref{xyz}), the upper limit of integration with respect to $v_{n-3}$ becomes $v_{n-4}/\sqrt{1+v_{n-4}^2}$ in terms of the new variables.  For the limits of integration with respect to the new variables, we first obtain
$$\frac{y_{n-1}}{\sqrt{1-3y_{n-1}^2}}\leq v_{n-4}\leq \frac{1}{\sqrt{n-4}}.$$
Here the lower bound follows directly from $L_{n-4}\le y_{n-4}$ and (\ref{Lieq}) with $i=n-4$.  To obtain the upper bound, we first note that $y_2\le 1$, by (\ref{yieq}) and (\ref{omdef33}), and that this is equivalent to $2y_2^2+y_3^2+\cdots+y_{n-4}^2\le |Y_{n-4}|^2$.  Expressing this inequality in terms of the new variables, we see that the region of integration is contained in the ellipsoid
\begin{equation}\label{newelli}
2v_2^2+v_3^2+\dots +v_{n-4}^2\leq 1,
\end{equation}
from which the upper bound follows directly. Now, once $v_{n-4}, v_{n-5},\dots,v_{i+1}$ have been fixed, we find that
$$v_{i+1}\le v_{i}\le \left(\frac{1-v_{n-4}^2-\cdots-v_{i+1}^2}{i}\right)^{1/2},$$
for $i=2,\dots,n-3$.  Here the lower bound results from (\ref{omdef33}) and the changes of variable via (\ref{yieq}), while the upper bound is again a consequence of (\ref{newelli}) and the fact that to maximize $v_i$, one must take $v_2=v_3=\cdots=v_{i+1}=v_i$ to reach the boundary of the ellipsoid. Thus the integral in (\ref{xyz}) becomes
\begin{equation}\label{ZKVin}
\int_{\frac{y_{n-1}}{\sqrt{1-3y_{n-1}^2}}}^{\frac{1}{\sqrt{n-4}}}Z_KV \,dv_{n-4},
\end{equation}
where
\begin{equation}\label{ZKdef}
 Z_K=Z_K(v_{n-4},m,y_{n-1}) = \int_{v_{n-4}}^{\left(\frac{1-v_{n-4}^2}{n-5}\right)^{1/2}}\cdots \int_{v_{3}}^{\left(\frac{1-v_{n-4}^2-\cdots-v_3^2}{2}\right)^{1/2}} h_K(M)\,dv_2\cdots dv_{n-5},
\end{equation}
$$M=\frac{\left((1-v_2^2-\cdots-v_{n-4}^2)^{1/2}, v_2, \dots, v_{n-4}, \frac{m}{\sqrt{1-m^2}}, \frac{y_{n-1}}
{\sqrt{1-2y_{n-1}^2}\ \sqrt{1-m^2}}, \frac{y_{n-1}}
{\sqrt{1-2y_{n-1}^2}\ \sqrt{1-m^2}},0\right)}{(1-v_2^2-\cdots-v_{n-4}^2)^{1/2}},$$ and
$$
 V=V (v_{n-4},y_{n-1})=\int_{\frac{y_{n-1}}{\sqrt{1-2y_{n-1}^2}}} ^{\frac{v_{n-4}}{\sqrt{1+v_{n-4}^2}}}  (1-v_{n-3}^2)^\frac{n-5}{2} T\, dv_{n-3}.
$$

We claim that
\begin{equation}\label{PhiMean}
\int_{\frac{y_{n-1}}{\sqrt{1-3y_{n-1}^2}}}^{\frac{1}{\sqrt{n-4}}}Z_{C^n}V\,dv_{n-4}= 0,
\end{equation}
where $C^n$, as in Lemma~\ref{lemdiff}, satisfies $h_{C^n}(x)=x_1$ for all $x=(x_1,\dots,x_n)$ such that $0=x_n\le x_{n-1}\le \cdots\le x_1\le 1$.  Indeed, we have $h_{C^n}(Y_{n-3},L_{n-2},L_{n-2},0)=1$ in view of (\ref{Yeq}) with $i=n-3$.  From this, it is clear that the left-hand side of (\ref{n5case}), and hence the integral in (\ref{ZKVin}), vanishes when $K=C^n$.  This proves the claim.

We know that $T\le 0$ for $v_{n-3}\le m(y_{n-1})$ and $T\ge 0$ for $v_{n-3}\ge m(y_{n-1})$.  It follows that $V\le 0$ and is decreasing if $v_{n-4}/\sqrt{1+v_{n-4}^2}\le m(y_{n-1})$, and hence when
$$\frac{y_{n-1}}{\sqrt{1-3y_{n-1}^2}}\le v_{n-4}\le \frac{m(y_{n-1})}{\sqrt{1-m(y_{n-1})^2}}.$$
For larger values of $v_{n-4}$, $V$ is increasing with respect to $v_{n-4}$ and so must become positive in order to satisfy (\ref{PhiMean}).  Consequently, there exists a function $q=q(m,y_{n-1})$ such that $V\le 0$ if $v_{n-4}\le q$ and $V>0$ if $v_{n-4}\ge q$.  Assuming that $Z_K/Z_{C^n}$ is an increasing function of $v_{n-4}$, we could then write
\begin{align*}
\int_{\frac{y_{n-1}}{\sqrt{1-3y_{n-1}^2}}}^{\frac{1}{\sqrt{n-4}}}Z_KV\,dv_{n-4}
&= -\int_{\frac{y_{n-1}}{\sqrt{1-3y_{n-1}^2}}}^{q} Z_{C^n} \frac{Z_K}{Z_{C^n}}|V | \, dv_{n-4} + \int_{q}^{\frac{1}{\sqrt{{n-4}}}}Z_{C^n} \frac{Z_K}{Z_{C^n}}V\,dv_{n-4} \\
  & \ge \frac{Z_K(q, m, y_{n-1})}{Z_{C^n}(q, m, y_{n-1})}\int_{\frac{y_{n-1}}{\sqrt{1-3y_{n-1}^2}}}^{\frac{1}{\sqrt{n-4}}} Z_{C^n}(v_{n-4}, m, y_{n-1}) V \, dv_{n-4}=0,
  \end{align*}
thus completing the proof of the lemma.

It remains to show that $Z_K/Z_{C^n}$ is an increasing function of $v_{n-4}$, for $n\ge 6$. Let $L$ be the $(n-3)$-dimensional convex body with support function
\begin{align*}
\MoveEqLeft{
  h_L(x_1,\dots, x_{n-3})}\\
  &= h_K\left(x_1, \dots, x_{n-4}, \frac{m}{\sqrt{1-m^2}}, \frac{y_{n-1}}
{\sqrt{1-2y_{n-1}^2}\ \sqrt{1-m^2}}, \frac{y_{n-1}}
{\sqrt{1-2y_{n-1}^2}\ \sqrt{1-m^2}}, x_{n-3}\right).
\end{align*}
Since $h_L$ is defined by fixing three coordinates in $h_K$, it is invariant under exchanges of the other coordinates.  Therefore $L$ has the symmetries of the coordinate cube $Q^{n-3}$.

By (\ref{jj}), we have
\begin{align*}
\MoveEqLeft{J_L(x_{n-4})=\frac{\int_{S(x_{n-4})}h_L\left(1,\frac{x_2}{x_1},\dots, \frac{x_{n-4}}{x_1},0\right)\,dx_2\cdots dx_{n-5}}{\int_{S(x_{n-4})}\,dx_2\cdots dx_{n-5}}} \\
  =& \frac{\int_{S(x_{n-4})}\frac{1}{x_1}h_K\left(x_1,x_2, \dots, x_{n-4}, \frac{m}{\sqrt{1-m^2}},\frac{y_{n-1}}
{\sqrt{1-2y_{n-1}^2}\ \sqrt{1-m^2}}, \frac{y_{n-1}}
{\sqrt{1-2y_{n-1}^2}\ \sqrt{1-m^2}},0\right)\,dx_2\cdots dx_{n-5}}{\int_{S(x_{n-4})}\frac{1}{x_1} x_1\,dx_2\cdots dx_{n-5}}\\
 =&Z_K(x_{n-4}, m, y_{n-1})/Z_{C^n}(x_{n-4}, m, y_{n-1}),
\end{align*}
where the previous equality follows from (\ref{ZKdef}) on noting that the limits of integration there coincide with (\ref{inarray}) with $n$ replaced by $n-3$. Moreover, since $L$ has the symmetries of $Q^{n-3}$, there is a $t>0$ such that $tL$ satisfies the hypotheses of Lemma~\ref{leman} with $n$ replaced by $n-3$.  In view of the obvious facts that $J_{tK}=tJ_K$ and that since $n\ge 6$, we have $3\le n-3<n$, we can appeal to the inductive hypothesis to conclude that $J_L(x_{n-4})$ is an increasing function of $x_{n-4}$.  It follows that $Z_K(v_{n-4}, m, y_{n-1})/Z_{C^n}(v_{n-4}, m, y_{n-1})$ is an increasing function of $v_{n-4}$ and the lemma is proved.  \qed


\begin{thebibliography}{999}

\bibitem{BalB12}
P.~Balister and B.~Bollob\'{a}s, Projections, entropy and sumsets, {\em Combinatorica} {\bf 32} (2012), 125–-141.

\bibitem{Ben99}
A.~Ben-Israel, The change of variables formula using matrix volume, {\em SIAM J. Matrix Anal.} {\bf 21} (1999), 300--312.

\bibitem{BetM83}
U.~Betke and P.~McMullen, Estimating the sizes of convex bodies
from projections, {\em J.~London Math. Soc.} (2) {\bf 27} (1983), 525--538.

\bibitem{BCT06}
J.~Bennett, A.~Carbery, and T.~Tao, On the multilinear restriction and Kakeya conjectures, {\em Acta Math.} {\bf 196} (2006), 261–-302.

\bibitem{BurZ88}
Y.~D.~Burago and V.~A.~Zalgaller, {\em Geometric Inequalities},
Springer, New York, 1988.

\bibitem{CCG}
S.~Campi, A.~Colesanti, and P.~Gronchi,
Convex bodies with extremal volumes having prescribed brightness in finitely many directions, \emph{Geom. Dedicata} {\bf 57} (1995), 121--133.

\bibitem{CamG11}
S.~Campi and P.~Gronchi, Estimates of Loomis-Whitney type for intrinsic volumes, \emph{Adv. Appl. Math.} {\bf 47} (2011), 545--561.

\bibitem{ConB74}
D.~R.~Conant and W.~A.~Beyer, Generalized Pythagorean theorem, \emph{Amer. Math. Monthly} {\bf 81} (1974), 262--265.

\bibitem{Fir60}
W.~J.~Firey,  Pythagorean inequalities for convex bodies, {\em Math. Scand.}
{\bf 8} (1960), 168--170.

\bibitem{Gar06} R.~J.~Gardner, {\em Geometric
Tomography}, second edition, Cambridge University Press, New York, 2006.

\bibitem{Garling07}
D.~J.~H.~Garling, {\em Inequalities: a journey into linear analysis}, Cambridge University Press, Cambridge, 2007.

\bibitem{Gro08}
M.~Gromov, Entropy and isoperimetry for linear and non-linear group actions, {\em
Groups Geom. Dyn.} {\bf 2} (2008), 499–-593.

\bibitem{Gru07}
P.~M.~Gruber, {\em Convex and Discrete Geometry}, Springer, Berlin, 2007.

\bibitem{GMR10}
K.~Gyarmati, M.~Matolcsi, and I.~Z.~Ruzsa, A superadditivity and submultiplicativity property for cardinalities of sumsets, {\em Combinatorica} {\bf 30} (2010), 163–-174.

\bibitem{Had57}
H.~Hadwiger, {\em Vorlesungen \"{u}ber {I}nhalt, {O}berfl\"{a}che
und {I}soperimetrie}, Springer, Berlin, 1957.

\bibitem{Han78}
T.~S.~Han, Nonnegative entropy measures of multivariate symmetric correlations, {\em Information and Control} {\bf 36} (1978), 133–-156.

\bibitem{HLP} G.~H.~Hardy, J.~E.~Littlewood, and G.~P\'{o}lya,
{\em Inequalities}, Cambridge University Press, Cambridge, 1959.

\bibitem{HugS11}
D.~Hug and R.~Schneider,  Reverse inequalities for zonoids and their application, {\em Adv. Math.} {\bf 228} (2011), 2634–-2646.

\bibitem{LooW49}
L.~H.~Loomis and H.~Whitney,  An inequality related to the
isoperimetric inequality, {\em Bull. Amer. Math. Soc.} {\bf 55} (1949), 961--2.

\bibitem{LYZ1}
E.~Lutwak, D.~Yang, and G.~Zhang, A new ellipsoid associated with convex bodies, {\em Duke Math. J.} {\bf 104} (2000), 375--390.

\bibitem{LYZ2}
E.~Lutwak, D.~Yang, and G.~Zhang, Volume inequalities for isotropic measures, {\em Amer. J. Math.} {\bf 129} (2007), 1711--1723.


\bibitem{MSW09}
N.~Madras, D.~W.~Sumners, and S.~G.~Whittington, Almost unknotted embeddings of graphs in $\Z^3$ and higher dimensional analogues, {\em J. Knot Theory Ramifications} {\bf 18} (2009), 1031–-1048.

\bibitem{Mey88} M.~Meyer, A volume inequality concerning sections of convex
sets, {\it Bull. London Math. Soc.} {\bf 20} (1988), 151--155.

\bibitem{NPRR}
H.~Q.~Ngo, E.~Porat, C.~R\'{e}, and A.~Rudra, Worst-case optimal join algorithms, arXiv:1203.1952v1.

\bibitem{Pfe93}
W.~F.~Pfeffer, {\em The {R}iemann {A}pproach to {I}ntegration:
{L}ocal {G}eometric {T}heory}, Cambridge University Press, New York,
1993.

\bibitem{SchW12}
F.~E.~Schuster and M.~Weberndorfer, Volume inequalities for asymmetric Wulff shapes, {\em J. Differential Geom.} {\bf 92} (2012), 263--283.

\bibitem{Sch93}
R.~Schneider, {\em Convex Bodies: The Brunn-Minkowski Theory}, Cambridge University Press, Cambridge, 1993.

\bibitem{She64}
G.~C.~Shephard, Shadow systems of convex bodies, \emph{Israel J. Math.} {\bf 2} (1974), 229--236.

\end{thebibliography}
\end{document}